\documentclass{article}[ht]
\usepackage{graphicx}
\usepackage[utf8]{inputenc}
\usepackage[twoside,letterpaper,margin=1in]{geometry}
\usepackage{microtype}
\usepackage{graphicx,url,amsmath,amsthm,amsfonts,amssymb,subfigure,bbm,bm,enumitem}
\usepackage{hyperref}
\usepackage{graphics,epic,eepic,color}
\usepackage{tikz}       
\usepackage{pgfplots}    
\usepackage{sidecap}
\pgfplotsset{compat=1.18}
\usepackage{comment}

\newtheorem{theorem}{Theorem}
\newtheorem{conjecture}[theorem]{Conjecture}
\newtheorem{proposition}[theorem]{Proposition}

\newtheorem{lemma}[theorem]{Lemma}
\newtheorem{corollary}[theorem]{Corollary}
\newtheorem{question}[theorem]{Question}

\DeclareMathOperator{\Cr}{\mathrm{cr}}
\DeclareMathOperator{\conv}{conv}
\DeclareMathOperator{\cone}{pos}
\DeclareMathOperator{\sph}{sph}

\DeclareMathOperator{\area}{area}

\DeclareMathOperator{\var}{var}

\DeclareMathOperator{\sign}{sign}

\title{Crossing numbers of simplicial complexes}
\author{Alfredo Hubard\footnote{Universit\'e Gustave Eiffel, CNRS, LIGM, Marne-la-Vall\'ee, France. alfredo.hubard@univ-eiffel.fr}, Cheikh Saliou Ndiaye\footnote{Universit\'e Gustave Eiffel, Univ Paris Est Creteil, CNRS, LAMA UMR8050, F-77447 Marne-la-
Vallée, France. cheikh-saliou.ndiaye@univ-eiffel.fr} and Oriol Solé-Pi\footnote{Departmet of Mathematics, Massachusetts Institute of Technology, Cambridge, MA, US. oriolsp@mit.edu}}

\date{}

\begin{document}
\maketitle
\begin{abstract}This paper studies crossing numbers of $d$-dimensional simplicial complexes drawn in $\mathbb R^{2d}$ or $\mathbb S^{2d}$.
We clarify and generalize a number of known results and conjectures from the graph case to higher dimensions: the crossing lemma (including the exceptional case of two-dimensional complexes drawn in dimension $4$), a generalized Harary-Hill conjecture (and evidence for it), existence of a midrange crossing constant (in a weak sense), and upper bound constructions based on stochastic geometry (related to Sylvester's problem) and the cyclic polytope. 

Some of our results are new even in the case  of graphs: we give a simple proof of the Harary-Hill conjecture for antipodal drawings, and sharpen the threshold above which the rectilinear and spherical crossing numbers of the Erd\H{o}s-R\'enyi random graph behave like those of the complete graph. 
\end{abstract}

\section{Introduction}

The crossing number is one of the best studied measures of non-planarity for graphs. Several variations of the crossing number have been considered in the literature, and countless results have been obtained from both the algorithmic and theoretical points of view. We refer the reader to~\cite{schaefersurvey} and~\cite{thirteenproblems} for an overview of the literature on crossing numbers, as well as multiple interesting open problems in this area. 

In this paper, we study crossing numbers of simplicial complexes, for which our understanding is considerably poorer. 

A main theme of the paper is that the best constructions we are aware of for graphs come from stochastic geometry, and this is also the case in higher dimensions.  

By the dimension of a simplicial complex $X$ we mean the greatest dimension of one of its faces. For a simplicial complex $X$ of dimension $d$ and a continuous map $f \colon X \to \mathbb R^{2d}$, the \textit{crossing number} of $f$, denoted by $\Cr(f)$, is the number of self-intersections of $f$. The \textit{crossing number} of $X$, denoted by $\Cr(X)$, is the minimum of $\Cr(f)$ when taken over all piecewise linear generic maps $f \colon X \to \mathbb R^{2d}$. We will also discuss:
\begin{itemize}
    \item  the \emph{independent pair crossing number}, denoted by $\Cr_\iota(X)$, for which instead of counting crossings we count pairs of vertex-disjoint $d$-faces whose images intersect,
    \item the \emph{rectilinear}  crossing number, denoted by $\Cr_{\mathbb R^{2d}}$, in which we restrict the problem to facewise affine maps of $X$ to $\mathbb R^{2d}$,
    \item and the \emph{spherical} crossing number, denoted by $\Cr_{\mathbb S^{2d}}$, for which we restrict the problem to spherically geodesic maps from $X$ to $\mathbb S^{2d}$
    \end{itemize}
    We call the last two variants the \emph{geometric crossing numbers}. Many non-quantitative results that are valid for one of the two geometric crossing numbers are valid for both with the same proof. For briefness, we sometimes only state one of the two.

    In every case we assume that $f$ is a generic piecewise linear map\footnote{In the case of $\mathbb S^{2d}$, the cone from the origin of the image of every face is piecewise linear and the map is piecewise geodesically spherical}. 
    Generic here means that every self intersection of $f(X)$ is transverse and isolated from the rest, and that for every point $y$ in the target space the set $f^{-1}(y)$ contains at most two points. This, in particular, guarantees that $\Cr(f)$ is well defined.

The following inequalities are immediate from the definitions. 

\[\Cr_\iota(X) \leq  \Cr(X) \leq  \Cr_{\mathbb S^{2d}}(X)\leq \Cr_{\mathbb R^{2d}}(X).\]

We denote the set of $i$-dimensional faces of $X$ by $X(i)$, and we write $X_{(i)}$ for the $i$-skeleton of $X$, i.e., the simplicial complex that consists of all the faces of $X$ up to dimension $i$. 

In this paper we focus on the  extremal problem: what is the smallest crossing number of a $d$-dimensional simplicial complex $X$ with a prescribed number of $d$-faces, $|X(d)|$, and $(d-1)$-faces, $|X(d-1)|$?
It turns out that, provided that $\frac{|X(d)|}{|X(d-1)|}$ is above a certain constant (that depends only on the dimension), the crossing number of every simplicial complex  grows at least like $\frac{|X(d)|^{d+2}}{|X(d-1)|^{d+1}}$.  Given a sequence of complexes $\mathcal X=\{X_n\}_{n=1}^\infty$, we adopt the following notation $$B(\mathcal X):=\lim_n \Cr(X_n) \frac{|X_n(d-1)|^{d+1}}{|X_n(d)|^{d+2}}$$ whenever the limit exists. We call the constant $B(\mathcal X)$ the \emph{crossing constant} of $\mathcal {X}$. We denote the independent-pair, spherical and rectilinear analogs by $B^\iota$, $B^{\mathbb S}$, $B^{\mathbb R}$ and call them the pair crossing constant, the spherical crossing constant and the rectilinear crossing constant.

We organize the discussion in three distinct regimes:

\begin{itemize}
\item 
the sparse regime: $|X(d)|= \Theta(|X(d-1)|)$. 
\item the dense regime: 
$|X(d)|=\Theta(|X(d-1)|^{\frac{d+1}{d}})$.
\item the midrange: roughly, this is everything that is neither sparse nor dense. 
\end{itemize}

There is a vast literature on crossing number for the case of graphs ($d=1$). We begin by surveying the results on the aforementioned extremal question in the case of graphs from the point of view of the crossing constant. We then discuss higher dimensional generalizations for each regime.

\subsection{Crossing numbers of graphs}

The study of crossing numbers was pioneered by Tur\'an and Hill~\cite{turan1977note,HararyHill1963}, who respectively considered drawings of the complete bipartite graph and of the complete graph, and asked for the minimum possible number of crossings. The celebrated conjectures of  Zarankiewicz and Harary-Hill address these questions, respectively. Each conjecture was motivated by a drawing believed to be optimal, and both have generated much subsequent work. See~\cite{Székely2016} for an account of these problems' history and the state of the art from a decade ago, as well as~\cite{balogh2019closing,bojan2020_hill_conjecture,balogh2023crossing,brosch2024new,streltsova2025spherical} for some more recent progress.

The Harary-Hill~\cite{guy1967combinatorial,HararyHill1963} conjecture states that \[\Cr(K_n)=\frac{1}{4} \left\lfloor \frac{n}{2} \right\rfloor \left\lfloor \frac{n - 1}{2} \right\rfloor \left\lfloor \frac{n - 2}{2} \right\rfloor\ \left\lfloor \frac{n - 3}{2} \right\rfloor\,\sim \frac{n^4}{64}.\] 
Meanwhile, the Zarankiewicz conjecture~\cite{zarankiewicz1955problem} (see also~\cite{urbanik1955solution}) claims that the crossing number of the complete bipartite graph with $n$ vertices on each side\footnote{The Zarankiewicz conjecture addresses also the more general case of $\Cr(K_{n,m})$, but here we focus only on the balanced case.} is  \[\Cr(K_{n,n})=\left(\lfloor \frac{n}{2} \rfloor\lfloor \frac{n - 1}{2} \rfloor\right)^2  \sim \frac{n^4}{16}. \] 

Richter and Thomassen~\cite{RichterThomassen1997} showed that if the limit $\lim_{n \to \infty} \frac{1}{n^4}\Cr(K_{n,n})$ (which is known to exist) evaluates to $\frac{1}{16}$, then one also has that $\lim_{n \to \infty} \frac{1}{n^4} \Cr(K_n)=\frac{1}{64}$. 
Below we will discuss a variant of the crossing number in which edges are assumed to be drawn as rectilinear segments in the plane $\mathbb R^2$, and a variant in which edges are assumed to be drawn as geodesic arcs on the round sphere. The constructions of Zarankiewicz and Hill can be realized by rectilinear and spherical drawings, respectively. A general idea that underlies all that follows is that in many interesting cases $\Cr(X)=\Cr_\mathbb S(X)$. 
\footnote{When this manuscript was about to be made public, a solution to the Zarankiewicz and Hill conjectures was announced. Apart from the addition of this footnote, we have not revised the text to reflect this. }

In the 1970s, Erd\H{o}s and Guy made two conjectures about crossing numbers of arbitrary graphs. The first of these conjectures was solved in 1982 by Ajtai, Chvátal, Newborn and Szemerédi~\cite{ACNS1982} and, independently, by Leighton~\cite{Leighton1983}, giving rise to what is usually known as the \textit{crossing lemma}. Much work has gone into improving the constant in this result~\cite{PachToth1997, pach2006improving, ackerman2019topological,bungener2024improving,bungener2026five}. The version below is due to B{\"u}ngener, Franz, Kaufmann and Pfister~\cite{bungener2026five}.

\begin{theorem}\label{thm:crossing_lemma}The crossing number of any graph $X$ satisfying $|X(1)|\geq 7.39 |X(0)|$ is at least $\frac{1}{27.3} \frac{|X(1)|^3}{|X(0)|^2}$.
\end{theorem}

For $a\in[3,4]$, we understand the growth of the least crossing number among graphs $X$ with $|X(1)|\geq a|X(0)|$. Indeed, consider an $n$-vertex quadrangulation of the sphere (which, by Euler's formula, has $n-2$ quadrangles and $2n-4$ edges). Then, add a diagonal for each quadrangle and an $(a-3)$ fraction of the remaining diagonals of the same quadrangles. This gives rise to a graph $X$ with $|X(1)|=a|X(0)|-O(1)$ edges which has the least possible crossing number among all graphs with the same number of vertices and edges (i.e., $(a-3)(n-2)$ crossings). Moreover, as $n$ grows, the crossing number of $X$ is asymptotically $\frac{a-3}{a^3}\frac{|X(1)|^3}{|X(0)|^2}$. From results in~\cite{pach2006improving}, it follows that if $X$ is a graph with $|X(1)|\geq a |X(0)|$ then $\Cr(X) \geq \frac{7a-25}{3a^3}\frac{|X(1)|^3}{|X(0)|^2}$, and this is best possible for $4\leq a \leq 5$. It was also conjectured there that if $|X(1)|\geq a|X(0)|$ then $\Cr(X) \geq \frac{25a-105}{6a^3}\frac{|X(1)|^3}{|X(0)|^2}$; if true, this would be best possible in the regime $5 \leq a \leq 6$ due to a construction in the same paper.

For larger $a$ and for graphs with $n$ vertices and $a n$ edges, the graphs with the smallest crossing number that we are aware of are obtained by taking a section $[-k,k]^2\cap\mathbb Z^2$ of the integer lattice and connecting pairs of integer points at Euclidean distance at most $r$ for an appropriately chosen $r=r(a)$. We denote this example by $D_{r,k}$ in what follows, where usually $r=o(k)$ and $r$ is such that  $r^2$ is a sum of two squares. These graphs come with a natural drawing. Namely, one can simply perturb the vertex set slightly to ensure general position, and then draw every edge as a segment. (We note that it would be asymptotically equivalent but perhaps cleaner to work with a lattice on a flat torus).

For $s$ a positive integer and $t>0$, define \[\kappa_1(s,t):=\min_{X,|X(0)|=s, |X(1)|\geq t} \Cr(X)\] and \[b_1(a):=\lim_{s \to \infty} \kappa_1(s,a s) \frac{s^{2}}{(a s)^{3}}=\frac{1}{a^{3}} \lim_{s \to \infty}\frac{\kappa_1(s,a s)}{s}.\]
As we explain below, it is well known that $b_1(a)$ exists. Figure~\ref{fig:1} summarizes the aforementioned results. 

\begin{SCfigure}[][h]
\begin{tikzpicture}[scale=0.85]
\begin{axis}[
    domain=0:10,
    samples=1000,
    xmin=0, xmax=14,
    ymin=0, ymax=1/8,
    xlabel={$|X(1)|/|X(0)|$},
    ylabel={crossing constant},
    grid=both,
     yticklabel style={font=\tiny},
      xticklabel style={font=\tiny},
]

\addplot[
    blue,
    thick,
    domain=0:3
]
{0};

\addplot[
    blue,
    thick,
    domain=3:4
]
{(x-3)/(x^3)};

\addplot[
    cyan,
    thick,
    domain=4:14
]
{(x-3)/(x^3)};

\addplot[
    blue,
    thick,
    domain=4:5
]
{(7*x-25)/(3*x^3)};

\addplot[
    blue,
    thick,
    dashed,
    domain=5:6
]
{(25*x-105)/(6*x^3)};

\addplot[
    green!60,
    thick
]
coordinates {(7.39,1/27.3) (14,1/27.3)};

\addplot[
    only marks,
    mark=*,
    mark size=1pt,
    red
]
coordinates {
    (4,1/64)
    (6,0.04166667)
    (10,0.05500000)
    (12,0.06134259)
    (14,0.06851312)
 
};

\end{axis}
\end{tikzpicture}
\caption{The horizontal axis corresponds to half the average degree, $a$. The vertical axis to $b_1(a)$. The solid blue line describes $b_1(a)$ for $a\leq 5$, and the dashed blue line corresponds to the conjecture for $5<a<6$. The green horizontal line is the lower bound on $b_1(a)$ for $a\geq 7.39$ provided by Theorem~\ref{thm:crossing_lemma}. The cyan curve corresponds to the standard proof using Euler's formula. The red dots are upper bounds for $a=4,6,10,12,14$ which have been derived computationally from the graphs $D_{r,k}$ by perturbing their natural drawings slightly; the perturbations used are very unlikely to be optimal.\\ }\label{fig:1}

\end{SCfigure}
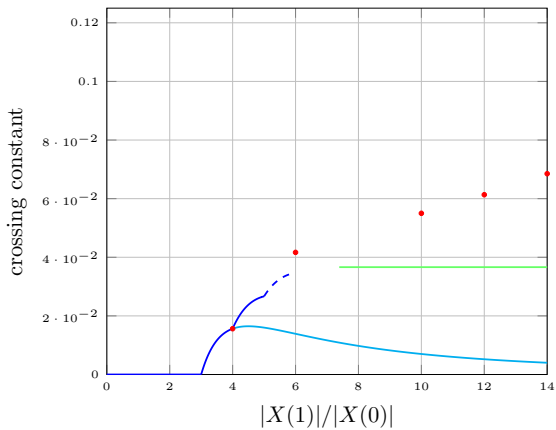

The second of the Erd\H{o}s-Guy conjectures from~\cite{ErdosGuy1973} was resolved by Pach, Spencer and T\'oth~\cite{pach1996new} in 1996. Before stating this result, define the \textit{midrange crossing constant} as
\[B_1=\limsup_{a \to \infty} b_1(a).\]

\begin{theorem}\label{midrange}\cite{pach1996new}
Along any sequence of pairs $(s,t)$ with $s \ll t \ll s^2$, $\lim_{s\to+\infty}\kappa_1(s,t)\frac{s^2}{t^3}=B_1$.
\end{theorem}

The best estimates for $B_1$ that we are aware of are $0.03663\sim \frac{1}{27.3}\leq B_1\leq \frac{8}{9\pi^2}\sim 0.0901$. While the details of the inequality $B_1\leq\frac{8}{9\pi^2}$ are left to the reader in~\cite{pach1996new,pach2006improving}, an elegant proof of this inequality, generalizing a construction of Moon, can be found in~\cite{czabarka2019some}.  Figure~\ref{fig:2} analyzes the speed of convergence of the grid examples (and their natural drawings) to the best known upper bound. 

\begin{SCfigure}[][h]
\begin{tikzpicture}[scale=0.8]
\begin{axis}[
    domain=0:10,
    samples=1000,
    xmin=4, xmax=800,
    ymin=0, ymax=1/8,
    xlabel={$|X(1)|/|X(0)|$},
    ylabel={crossing constant},
    grid=both,
     yticklabel style={font=\tiny},
      xticklabel style={font=\tiny},
]

\addplot[
    green!60,
    thick
]
coordinates {(7.39,1/27.3) (14,1/27.3)};

\addplot[
    teal,
    thick
]
coordinates {(0, {8/(9*pi^2)}) (1000, {8/(9*pi^2)})};

\addplot[
    orange,
    domain=4:800,
    samples=1000
]
{(8/(9*pi^2))*(1 - 4.05/x + 9.33/(x^2))};
\addplot[
    only marks,
    mark=*,
    mark size=0.5pt,
    red
]
coordinates {
    (4,1/64)
    (6,0.04166667)
    (10,0.05500000)
    (12,0.06134259)
    (14,0.06851312)
    (18.0, 0.07098765)
     (22.0, 0.07400451)
     (24.0, 0.07667824) 
     (28.0, 0.07826166)
     (30.0, 0.07885185)
     (34.0, 0.07950845)
     (40.0, 0.08214062)
     (44.0, 0.08230419)
     (48.0, 0.08265516)
     (50.0, 0.083056)
     (54.0, 0.08362546)
     (56.0, 0.08413812)
     (72, 0.0851)
};
\end{axis}
\end{tikzpicture}
\caption{The axes have the same meaning as in Figure~\ref{fig:1}. The red points are again the upper bounds on $b_1(a)$ (for certain integer values of $a$) which have been derived computationally from graphs of the form $D_{r,k}$. The orange curve is the quadratic function of $1/a$ that best fits the red points in the sense of least squares. The dark green line lies at height $\frac{8}{9\pi^2}$.\\ \\ \\  }\label{fig:2}
\end{SCfigure}
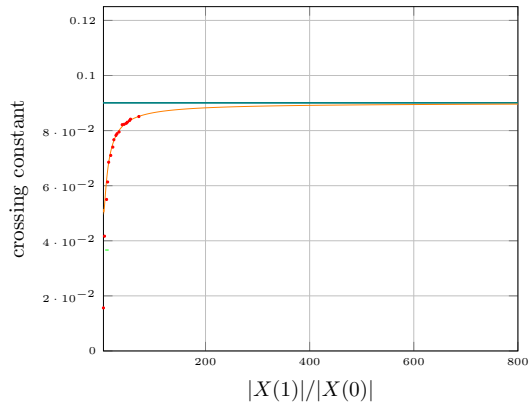

Let us define
\[b^+_1(\lambda):=\liminf_{s \to \infty} \kappa_1\left(s, \lambda{\binom s2} \right)\frac{8}{\lambda^3 s^4}.\]

In~\cite{Moon1965}, Moon observed that if the complete graph $K_n$ is drawn on the round sphere by sampling $n$ vertices uniformly and independently at random, and every pair of vertices is then connected using the shortest geodesic between them, then the expected crossing number is $\frac{3}{8}{\binom n4}\sim\frac{n^4}{64}$. Asymptotically, this coincides with the Harary-Hill conjecture.  The random geometric graph on the sphere
is constructed by drawing $n$ points independently at random from the uniform distribution on the round two-dimensional sphere, and then adding an edge between any two of them which are at distance no more than some parameter $r$. As before, edges are drawn by geodesic arcs. In~\cite{czabarka2019some}, it was shown that in this model the edge density is $\frac{1}{2}(1-\cos r)$ and the crossing density is $\frac{1}{8 \pi^2}(\sin r-r \cos r)^2$. This, in turn, implies that \[b^+_1(\lambda) \leq \frac{1}{\pi^2}\frac{(2\sqrt{\lambda-\lambda^2}-(1-2\lambda)\arccos(1-2\lambda))^2}{8 \lambda^3}.\]
When $\lambda=1$, one recovers Moon's bound, and when $\lambda \to 0$ we recover the bound $B_1 \leq \frac{8}{9 \pi^2}$. 

This is the first problem that we are aware of in which the geometric random graph (in fixed dimension) might yield the minimizer of a non-metric extremal problem. It would be remarkable if this construction from stochastic geometry turned out to yield the dense graph with the least crossing number for every density $\lambda \in (0,1]$, a problem in combinatorial topology. Figure 3 depicts the upper bound on $b_1^+(\lambda)$ that arises from the random geometric graph as $\lambda$ ranges from $0$ to $1$, and the best lower bound we could find combining results from
\cite{pach2006improving, montaron2005improvement,balogh2019closing,bungener2026five}.

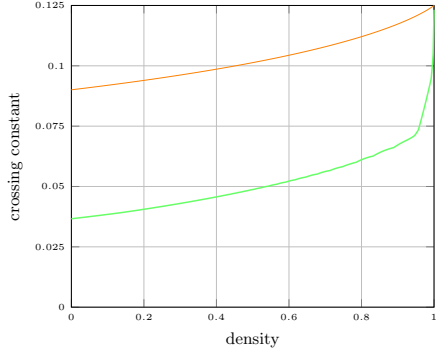
\begin{SCfigure}[][ht]
\begin{tikzpicture}[scale=0.7]
\begin{axis}[
    xmin=0, xmax=1,
    ymin=0, ymax=1/8,
    xlabel={density},
    ylabel={crossing constant},
    grid=both,
    ytick={0,0.025,0.05,0.075,0.1,0.125},
    yticklabel style={
        font=\tiny,
        /pgf/number format/fixed,
        /pgf/number format/precision=3
    },
    xticklabel style={font=\tiny},
    label style={font=\small},
]
\addplot[orange]
table[x=lambda,y=upper,col sep=comma]
{dense_graph_lower.csv};

\addplot[green!60,thick]
table[x=lambda,y=lower,col sep=comma]
{dense_graph_lower.csv};
\end{axis}
\end{tikzpicture}
\caption{The upper bound (orange) for $b^+_1(\lambda)$ is derived from the spherical random geometric graph. At $\lambda=1$ it equals Moon's bound of $\frac{1}{8}$, and as $\lambda\to0$ it tends to $\frac{8}{9\pi^2}$, the best known upper bound on the midrange crossing constant. The lower bound (green) combines results of~\cite{pach2006improving,montaron2005improvement,balogh2019closing,bungener2026five}.\\ \\ }
\label{fig:3}
\end{SCfigure}

Notice from Figure~\ref{fig:2} that if the graphs $D_{r,k}$ (along with their natural drawings) aren't far from optimal, obtaining a close-to-optimal lower bound for the midrange crossing constant following the kind of techniques used in~\cite{pach2006improving,montaron2005improvement,ackerman2019topological,bungener2024improving,bungener2026five} will require extending our understanding of topological graphs to the case where the average degree is around $400$.

We now make explicit a well-known connection of crossing numbers with stochastic geometry. 
Given a probability measure $\mu$ in the plane such that any line has measure zero, Sylvester's problem asks for the probability $q(\mu)$ that four random points drawn independently according to $\mu$ are in convex position. Scheinerman and Wilf~\cite{Rectilinear_ScheinermanWilf1994} observed that 
\[
\lim _{n\to+\infty} \Cr_{\mathbb R^2}(K_n)\frac{1}{{\binom n4}}=\inf_{\mu} q(\mu),
\] where $\operatorname{cr}_{\mathbb R^2}(K_n)$ is the rectilinear crossing number of the complete graph $K_n$ and the infimum is taken among the family of measures that do not charge lines.

Sylvester's problem has been much studied in the case where the probability is uniform on a convex body. 
Blaschke~\cite{Blaschke1923} showed that, among convex bodies, the probability of four points being in convex position is maximized by the round disk and minimized by the triangle (and by the square among centrally symmetric convex bodies). 
Because of this, it might be tempting to use the distance functions induced by either the $|\cdot|_\infty$ or the $|\cdot|_1$ norms when defining $D_{r,k}$. 
However, using either of these distance functions in the definition of $D_{r,k}$ results in a slightly larger number of crossings. 
In fact, let us consider a uniform distribution on $[0,1]^2$ and a centrally symmetric convex set $K$ which induces the norm $|.|_K$. By normalizing appropriately, we can ensure that $\area(tK)=t^2$ for all $t>0$.
Now, define \[q_{\operatorname{mid}}(K):=\lim_{t \to 0} \frac{\mathbb P_{x_1,x_2,x_3,x_4}(\textrm{the segments }x_1x_3\textrm{ and }x_2x_4\textrm{ cross and } |x_i-x_{i+2}|_K<t)}{t^6}, \]
where indices are understood modulo $4$ and $x_1,x_2,x_3,x_4$ are i.i.d.\ uniform points from $[0,1]^2$. 
It is not hard to show that for any convex set $K$, $B_1 \leq q_{\operatorname{mid}}(K)$, and it turns out that $q_{mid}(K)$ is minimized precisely when $K$ is affinely equivalent to a disk, in which case it takes the value $\frac{8}{9\pi^2}$~\cite{lepsveridze2026uniform}.

\subsection{A crossing lemma for simplicial complexes}

A \textit{simplicial complex} $X$ is a family of finite subsets of a fixed set such that if $\sigma$ is in $X$ and $\tau$ is a subset of $\sigma$, then $\tau$ is also in $X$ (in other words, $X$ is \textit{downwards-closed}). Suppose that the vertex set of $X$ is $[n]$ and consider the standard basis $e_0,e_1, \ldots e_{n-1}$ of $ \mathbb R^n$. For each $\sigma \in X$, consider the set $\conv(\{e_i\}_{i \in \sigma})$ (indices are taken mod $n$) equipped with the topology it inherits as a subset of $\mathbb R^n$. This topological space is the \textit{geometric realization} of $X$. In what follows, we use $X$ to denote both the simplicial complex and its geometric realization. The subsets in $X$ are also called \textit{faces} or \textit{simplices} of $X$. The \textit{dimension} of a face is the number of elements it contains minus one, and the \textit{dimension} of $X$ is the largest dimension among its faces. Faces of dimension $0$ (i.e., singletons) are called \textit{vertices}, and faces which are not properly contained in some other face are called \textit{facets}.
The $i$\textit{-skeleton} of $X$, denoted by $X_{(i)}$, is the simplicial complex that has the same faces as $X$ up to dimension $i$, and no faces of dimension higher than $i$. We denote the set of $i$-faces (i.e., faces of dimension $i$) of $X$ by $X(i)$. The degree of a face $\tau$ with respect to the $k$-faces, denoted by $\deg_k(\tau)$, is the number of $k$-faces of $X$ having $\tau$ as a subface. The simplicial complex that contains all the subsets of a set of cardinality $n+1$, denoted by $\Delta^n$, is called the $n$\textit{-simplex}. 

In this paper, we focus on the case of $d$-dimensional complexes drawn in $\mathbb R^{2d}$ (or, equivalently, $\mathbb S^{2d}$). Whenever we write $\Cr(X)$ we assume that we are working with piecewise linear maps from $X$ to either $\mathbb R^{2d}$ or $\mathbb S^{2d}$, where $d$ is implicitly assumed to be the dimension of $X$.  In this case, the co-dimensions of two facets add up to the host dimension $2d$. Hence, generically, intersections between facets have dimension $0$. The case $d=1$ corresponds to the classical crossing numbers of graphs drawn on the plane. For any topological space $Y$, a map $f:X\to Y$ is called an \textit{embedding} if it is a homeomorphism between $X$ and its image $f(X)$.

Crossing numbers of simplicial complexes and similar extremal problems in combinatorial geometry and combinatorial topology have been considered in the past.  Dey and Pach~\cite{dey_pach1998} studied the maximum number of facets that an affinely embedded simplicial complex of some fixed dimension may have without forcing certain configurations to appear. In Chapter 8 of her PhD thesis, Gundert (under the supervision of Wagner)~\cite{gundert2013expansion} gave a lower bound for the affine crossing number of certain high dimensional expanders. Later, Wagner and Wild~\cite{wagner2023coboundary} provided a lower bound on the crossing number of a simplicial complex in terms of the high-dimensional expansion properties of its so called \textit{deleted join}. See also~\cite{wagner2013} for multiple related results and questions.

In~\cite{adiprasito2018combinatorial}, a conjecture concerning crossing numbers of simplicial complexes (referred to as the Gr\"unbaum-Kalai-Sarkaria Conjecture) was proven. The relations between the different variants of this conjecture (see Conjecture 27 in~\cite{kalai2002algebraic}, Conjecture 1 in~\cite{SARKARIA1987}, Theorem 3 in~\cite{sarkaria} and the last sentence of~\cite{GRUNBAUM_1970_higher_four_color}) are explained in~\cite{gundert2018complexity}. 
Roughly speaking, these conjectures claim that the $d$-skeleton of the boundary of the $(2d+1)$-dimensional cyclic polytope maximizes the ratio $|X(d)|/|X(d-1)|$ among complexes with crossing number zero. This conjecture generalizes McMullen's celebrated Upper Bound Theorem for polytopes~\cite{McMullen1970} and Stanley's generalization to simplicial spheres~\cite{Stanley}.
We refer to the aforementioned references for precise statements of these results and their relation to algebraic shifting. In the proof of the moreover part of Corollary~\ref{crossing_lemma} below, we will also use the cyclic polytope as the main building block in our construction.   

\begin{theorem}\label{Adi}\cite{adiprasito2018combinatorial} If $X$ is a simplicial complex of dimension $d$ with $|X(d)|> (d+2)|X(d-1)|$, then $\Cr(X)> 0$.
 \end{theorem}

Note that the case $d=1$ of this theorem is folklore, and follows easily from Euler's formula and double counting. The starting point of this paper is the following corollary of Theorem~\ref{Adi}.

\begin{corollary}\label{crossing_lemma} Let $X$ be a $d$-dimensional simplicial complex with $|X(d)|\geq a|X(d-1)|$, for $a>0$. \\
1) If $d\neq 2$ then 
\[\Cr(X)\geq \frac{a-(d+2)}{a^{d+2}}\cdot\frac{|X(d)|^{(d+2)}}{|X(d-1)|^{(d+1)}}.\] 
2) Moreover, still for $d\neq 2$, if $d+2 <  a\leq d+2+\frac{2}{d+1}$, this is best possible (up to lower order terms) as $|X(d-1)|\to \infty$.
3) For any $d\ge 1$, one has
$\Cr_{\mathbb R^{2d}}(X) \geq \Cr_{\mathbb S^{2d}}(X)\geq \frac{a-(d+2)}{a^{d+2}}\cdot\frac{|X(d)|^{(d+2)}}{|X(d-1)|^{(d+1)}}.$ 
\end{corollary}

\noindent{\textbf{\em Remark:}} The expression $\frac{a-(d+2)}{a^{d+2}}$ achieves its maximum for $a= d+3+\frac{1}{d+1}$. Thus, if $d\neq 2$ and $a\geq d+3+\frac{1}{d+1}$, then \[\Cr(X)\geq \frac{(d+1)^{d+1}}{(d+2)^{2d+3}}\cdot  \frac{|X(d)|^{(d+2)}}{|X(d-1)|^{(d+1)}},\]
  and $\Cr_{\mathbb R^{2d}}(X) \geq \Cr_{\mathbb S^{2d}}(X)\geq \frac{(d+1)^{d+1}}{(d+2)^{2d+3}}\cdot  \frac{|X(d)|^{(d+2)}}{|X(d-1)|^{(d+1)}}$ for all $d$.

The case $d=1$ of this Corollary is the aforementioned \emph{crossing number inequality}. It was independently noticed by Dey~\cite{Dey1993}, Wagner~\cite{wagner2013} and others, that a version of Theorem~\ref{Adi} for $\Cr_\mathbb R$ or for $\Cr_\iota$  implies a version of  Corollary~\ref{crossing_lemma}. In~\cite{adiprasito2018combinatorial}, Corollary~\ref{crossing_lemma} is stated with the value $a=d+3$ (there, the reader is referred to~\cite{wagner2013}, where the result is not proved and the discussion is built around the Van Kampen class). 
As far as we are aware, no detailed proof of Corollary~\ref{crossing_lemma} can be found in the literature, and the case $d=2$ is exceptional, as we explain in detail below. 

For $d>2$, a variant of the Whitney trick (Lemma~\ref{whitney_trick} below) implies that in a drawing that minimizes crossings the star of every vertex is embedded. The introduction of the survey~\cite{schaefersurvey} includes a discussion on the difficulties and common mistakes involved in attempting to bypass an analogue of Lemma~\ref{whitney_trick} for variants of the crossing number for graphs.\footnote{The works~\cite{sarkaria, Nevo2007_HigherMinors, wagner2011_minors_random} and, in particular, \cite{wagner2013} propose an approach towards the Kalai-Sarkaria conjecture through the non-vanishing of the Van Kampen class. 
This easily yields Corollary~\ref{crossing_lemma} for all $d$, and this is what was meant in Theorem 3.12 of~\cite{wagner2013}.}
It turns out that we can use Theorem~\ref{Adi} (together with Lemma~\ref{whitney_trick}) for $d=3$ to obtain a crossing lemma also in the exceptional case $d=2$. 

\begin{proposition}\label{cased=2} Let $X$ be a $2$-dimensional simplicial complex with $|X(2)|\geq \frac{20}{3}|X(1)|$. 
\[\Cr(X)\geq \frac{3^3}{2^8 5^3}\cdot\frac{|X(2)|^{4}}{|X(1)|^{3}}.\] 
\end{proposition}
The decimal expression of $\frac{3^3}{2^8 5^3}$ is $0.00084375$.
 The key additional idea required in the proof of this proposition is that if $f \colon X \to \mathbb S^{4}$ is a map that sends vertex-disjoint $2$-faces to disjoint geometric simplices, then for any positive integer $q$ there exists a map $f \colon X* [q] \to \mathbb S^4 * \mathbb S^1=\mathbb S^{6}$ with the same property. Here, $[q]$ denotes the simplicial complex that has only $q$ disjoint points, and $*$ is the join operation; the definition of this operation will be given below.

 An interesting closely connected fact follows from the proof of Proposition 4.3 in~\cite{Melikhov2022}. Namely,  if $f_1 \colon X_1 \to \mathbb R^{2d_1}$, and $f_2 \colon X_2 \to \mathbb R^{2d_2}$, one can show that there exists a map $\hat{f_1*f_2}\colon X_1 * X_2 \to \mathbb R^{2(d_1+d_2+1)}$, which is built starting from the join $f_1*f_2\colon X_1 * X_2 \to \mathbb R^{2(d_1+d_2)+1}$, such that
$\Cr(\hat{f_1*f_2})\leq \Cr(f_1)\Cr(f_2)$. In particular, in this setting,
\[\Cr(X_1* X_2)\leq \Cr(X_1)\Cr(X_2).\]

\noindent{\textbf{\em Remark:}}  In a subsequent paper we will show that the statment of Corollary \ref{crossing_lemma} is valid for d=2.

\subsection{Complete and random complexes}

Recall that we write $\Delta^{n-1}$ to denote the simplex of dimension $n-1$, which has $n$ vertices. 

The next theorem generalizes Moon's bound to every $d$ and it is equivalent to an old result of Wendel~\cite{Wendel}. 

\begin{theorem}\label{Moon_Wendel}
  The limits 
 $\lim_{n \to \infty}\Cr_\iota(\Delta^{n-1}_{(d)}) \frac{{\binom nd}^{d+1}}{{\binom{n}{d+1}}^{d+2}}$ and $\lim_{n \to \infty}\Cr_{\mathbb S}(\Delta^{n-1}_{(d)}) \frac{{\binom{n}{d}}^{d+1}}{{\binom{n}{d+1}}^{d+2}}$ exist. Moreover,
 \[\frac{(d+1)^{(d+2)}d!}{(2d+3)!}\leq B^\iota(\Delta_{(d)})\leq B^\mathbb S(\Delta_{(d)})\leq   \frac{1}{2^{2d+2}}\frac{(d+1)^{d}}{d!},\] where $\Delta_{(d)}:=\{\Delta_{(d)}^{n}\}_{n=1}^\infty$.
 
\end{theorem}

For $d \to \infty$ the lower bound decays like $\frac{e}{8 \sqrt{2}} \frac{1}{d} \big( \frac{e}{4})^d$ and the upper bound like $\frac{e}{4 \sqrt{2\pi}} \frac{1}{\sqrt{d}} \big( \frac{e}{4})^d$.
The upper bound follows easily from a classic result of Wendel in geometric probability as we will explain. 
The lower bound for $\Cr(\Delta^{n-1}_{(d)})$ follows from double counting and work of Van Kampen~\cite{vanKampen1932} and, independently, Flores~\cite{Flores1933}, who showed that $\Cr_\iota(\Delta^{2d+2}_{(d)})=1$.\footnote{More generally, for the iterated join $K:=\Delta^{2k_1+2}_{(k_1)}*\Delta^{2k_2+2}_{(k_2)}*\ldots *\Delta^{2k_s+2}_{(k_s)}$ with $\sum_i (k_i+1)=d+1$, we have $\Cr(K)\neq 0$.}

The generalization of the complete bipartite graph to higher dimensions is the $(d+1)$-fold join of the discrete set with $n$ points $[n]$, denoted by $[n]^{*(d+1)}$. The join of two disjoint simplicial complexes $X$ and $Y$ is defined as \[X *Y:=\{\sigma \cup \tau: \sigma \in X, \tau \in Y\},\] where we use the convention that $\emptyset$ is in both $X$ and $Y$.  If $X$ is a simplicial complex, $X^{*k}$ denotes its $k$-fold join.  
This way, $[n]^{*(d+1)}$ is a $(d+1)$-partite complex. Van Kampen and Flores showed that $\Cr([3]^{*(d+1)})=1$, generalizing the non-planarity of the complete bipartite graph with $3$ vertices on each part. Write $[\infty]^{*(d+1)}:=\{[n]^{*(d+1)}\}_{n=1}^{\infty}$.

\begin{corollary}\label{double_count}
 The limits 
$\lim_{n \to \infty} \Cr_\iota\left([n]^{*(d+1)}\right) \frac{\big({(d+1)n^d\big)}^{(d+1)}}{\big (n^{(d+1)}\big)^{(d+2)}}$, $\lim_{n \to \infty} \Cr_\mathbb S\left([n]^{*(d+1)}\right) \frac{\big({(d+1)n^d\big)}^{(d+1)}}{\big (n^{(d+1)}\big)^{(d+2)}}$ 
exist. Moreover, we have that, \[B^\mathbb S\left([\infty]^{*(d+1)}\right)\leq (d+1)! B^\mathbb S(\Delta_{(d)})
,\] and   \[\left(\frac{d+1}{6}\right)^{d+1} \leq B^\iota([\infty]^{*(d+1)})\leq B^\mathbb S([\infty]^{*(d+1)})\leq  \left(\frac{d+1}{4}\right)^{d+1}.\] 
\end{corollary}

\noindent{\textbf{\em Remark:}}
Theorem~\ref{Moon_Wendel} and Corollary~\ref{double_count} hold for the crossing number provided $d\neq 2$. The case $d=1$ recovers the previously mentioned result of Richter and Thomassen from~\cite{RichterThomassen1997}.

\subsubsection{Spherical and rectilinear crossing numbers}

 Radon's lemma claims that a pointset $P$ with $2d+2$ points in dimension $2d$ in general position can be uniquely partitioned into $A\cup B$, $A\cap B=\emptyset$, and $\conv(A)\cap \conv(B)\neq \emptyset$. If we assume that $|A|\leq |B|$, then the type of the Radon partition is $|A|$. One endcase, $k=1$, corresponds to one point in the interior of the convex hull of the others, the remaining cases correspond to $P$ being in convex position. The case relevant to crossing numbers is the other endcase, where $k=d+1$.

Similarly, for a vector set $V=\{v_1,v_2, \ldots v_{2d+2}\}$ in $\mathbb S^{2d}\subset \mathbb R^{2d+1}$,  generically there exists a unique (up to multiplication by a non-zero scalar) $\lambda \in \mathbb R^{2d+2}$ such that $\sum_{i=1}^{2d+2} \lambda_iv_i=0$, and we define the spherical Radon partition by splitting $A:=\{v_i \in V: \lambda_i\geq 0\}$, $B:=\{v_i \in V: \lambda_i<0\}$.

Sylvester's problem can be generalized as follows. For a probability distribution $\mu$ on $\mathbb R^{2d}$ denote by $q(\mu)$ the probability that $2d+2$ points drawn independently from $\mu$ have Radon partition with $d+1$ points on each part. Analogously, for $\mathbb S^{2d}$, let $q(\mu)$ be the probability that $2d+2$ points drawn independently from $\mu$ have spherical Radon partition with $d+1$ points on each part.

Frick, Newman, and Pegden~\cite{fnp} studied this refinement of Sylvester's problem for independent Gaussian vectors, and showed that it is equivalent to Youden's demon problem — the problem of locating the mean among the order statistics of Gaussian points on the line — with the key tool being the fact that the Gale dual of independent Gaussian samples is again a Gaussian point set (with mean zero). 

The following is a straightforward generalization of Scheinerman and Wilf's observation. 

\begin{proposition} We have that
\[\lim_{n \to \infty} \frac{1}{{\binom{n}{2d+2}}}\Cr_{\mathbb R^{2d}} (\Delta_{(d)}^{n-1})=\inf_{\mu} q(\mu),\]
where the infimum is taken over measures $\mu$ on $\mathbb R^{2d}$ which do not charge affine hyperplanes. Also, 
\[\lim_{n \to \infty} \frac{1}{{\binom{n}{2d+2}}}\Cr_{\mathbb S^{2d}} (\Delta_{(d)}^{n-1})=\inf_{\mu} q(\mu),\] where the infimum is taken over measures on $\mathbb S^{2d}$ which do not charge hyperplanes.

\end{proposition}

Streltsova and Wagner~\cite{streltsova2025spherical}, building on earlier work of Wagner~\cite{LovaszVesztergombiWagnerWelzl2004, wagner2003ksets}, proved the Harary-Hill conjecture for spherical drawings. Their insights come partially from the Linhart-Eckhoff-Welzl conjecture, a generalization of McMullen's upper bound theorem, which we mentioned above and recall now in more detail. 
 
Consider the point set 
\[V_{n,m}:=\{v^i:=(i,i^2,i^3, \ldots i^{m})\in \mathbb N^{m}:i \in \{1,2,\ldots n\}\};\]
we call it a cyclic configuration and define the cyclic polytope on $n$ vertices in dimension $m$ as $C_{n,m}:=\conv(V_{n,m})$ (see~\cite{Stanley,Grunbaum,Matousek2002} for background on cyclic polytopes).  In what follows, we label the vertex $(i,i^2,\ldots i^m)$ simply by $i$. 
The Upper Bound Theorem claims that for each $s$, no convex polytope of $n$ vertices in $\mathbb R^m$ has more $s$-faces than the cyclic polytope $C_{n,m}$.

For a configuration $V\subset\mathbb R^m$ in linear general position, meaning that every $m$ vectors are linearly independent, consider the arrangement of hemispheres $\mathcal A(V):=\{H(v):v \in V\}$ where $H(v):=\{x \in \mathbb S^{m-1}:\langle x,v\rangle\leq 0\}$. Two points $x,y \in \mathbb S^{m-1}$ are on the same face of this arrangement whenever, for every $v \in V$, \[\sign(\langle x,v\rangle)=\sign(\langle y,v \rangle).\]

Given a face $\sigma \subset \mathbb S^{m-1}$ of the arrangement $\mathcal A(V)$, 
choose $x\in\sigma$ and let $s$ be the number of vectors in $V$ such that $\sign(\langle x,v\rangle)=0$; then, $\dim(\sigma)=(m-1)-s$. Let $t$ be the number of vectors in $V$ such that $\sign(\langle x,v\rangle)=-1$. We say that $t$ is the \emph{level} of the face $\sigma$. Denote by $f_{s,t}(V)$ the number of $(m-1)-s$ dimensional faces of level $t$. 

As a vector configuration, $V_{n,m}$ has the same sign patterns as $\{(1,i,\ldots,i^{m-1}):1\leq i\leq n\}$, since $(i,i^2,\ldots,i^m)=i(1,i,\ldots,i^{m-1})$. Thus, the arrangement formulation uses the cyclic polytope in affine dimension $m-1$. With this convention, the Upper Bound Theorem is equivalent to the following: for every $V\subset \mathbb R^m$ in linear general position with $|V|=n$ and every $s$ with $0\leq s \leq m-1$, $f_{s,0}(V)\leq f_{s,0}(V_{n,m})$. 
Write $f_{s,\leq t}:=\sum_{i=0}^t f_{s,i}$. Linhart, Eckhoff and Welzl independently conjectured the following.
\begin{conjecture}[Linhart-Eckhoff-Welzl]\label{gubt} Let $V\subset \mathbb R^m$ be a vector configuration in linear general position with $|V|=n$. For every integer $0\leq t \leq \frac{n-m-1}{2}$ and every integer $s$ with $0\leq s \leq m-1$, we have
\[f_{s,\leq t}(V)\leq f_{s,\leq t}(V_{n,m}).\]
\end{conjecture}

Following Streltsova and Wagner, we use Gale duality to show that the geodesic drawing on the alternating moment curve gives $\Cr_{\mathbb S^{2d}}(\Delta^{m-1}_{(d)})\leq\binom{\lfloor m/2\rfloor}{d+1}\binom{\lfloor(m-1)/2\rfloor}{d+1}$ for all $m$, and that the Linhart--Eckhoff--Welzl conjecture implies equality (Proposition~\ref{prop:alternating}). We derive the following:
\begin{proposition}\label{prop:lew-spherical}
    If the Linhart-Eckhoff-Welzl conjecture holds, we then have that
    \[ B^{\mathbb{S}}(\Delta_{(d)})=\frac{1}{2^{2d+2}} \frac{(d+1)^d}{d!}.\]
\end{proposition}

It is remarkable that both in the sparse regime $|X(d)|\sim C |X(d-1)|$ and the dense regime $|X(d)|\sim C |X(d-1)|^{\frac{d+1}{d}}$, the cyclic configurations and their Gale duals (called cocyclic configurations)  are conjecturally extremal. In the midrange the situation is less clear and we will discuss it in the next subsection. But, before that, we discuss another family of simplicial complexes of interest in the sparse regime and midrange, namely random simplicial complexes. 

\subsubsection{Geometric crossing numbers of Linial-Meshulam complexes}

Let $d\geq 1$. We consider, for every positive integer $n$ and any $p\in(0,1)$, the Linial-Meshulam random simplicial complex $\Delta_d(n,p)$. This is the random $d$-dimensional simplicial complex on $n$ labeled vertices whose $(d-1)$-skeleton is complete and in which each $d$-face is included independently with probability $p$. Note that for $d=1$, this corresponds precisely to the classical Erd\H{o}s-R\'enyi random graph $G(n,p)$. The subindex $d$ will be omitted from now on. 

Spencer and Tóth~\cite{SpencerToth2002} studied the crossing number of Erd\H{o}s-R\'enyi graphs and, as one of their main results, they showed that, for any sequence of probabilities which satisfies $p\gg \log n/n$, one has that \[\mathrm{cr}_{\mathbb R^2}(G)=p^2\mathrm{cr}_{\mathbb R^2}(K_n)+o(p^2n^4)\] with high probability as $n\rightarrow\infty$.
  We remark that their proof readily extends to spherical (instead of rectilinear) crossing numbers. Spencer and Tóth asked whether the condition that $p\gg\log n/n$ is necessary, and conjectured that $np\rightarrow \infty$ should actually suffice. Here, we prove and generalize this conjecture to all dimensions.

\begin{theorem}\label{thm:rectlinear_random_complex}
Fix \(d\geq 1\) and let \(X\sim \Delta_d(n,p)\).
If\[p n^d\rightarrow \infty,\] then \[\operatorname{cr}_{\mathbb R^{2d}}(X)=p^2\operatorname{cr}_{\mathbb R^{2d}}(\Delta_{(d)}^{n-1})+o_{\mathbb P}\bigl(p^2 n^{2d+2}\bigr).\] Here, $o_{\mathbb P}\bigl(p^2 n^{2d+2}\bigr)$ denotes an error term which, when divided by $p^2n^{2d+2}$, goes to $0$ in probability.

\end{theorem}

It seems very likely that the $pn^d\rightarrow\infty$ requirement is tight. Indeed, this is the threshold at which the ``giant connected component'' spanned by the $d$-dimensional faces becomes of size $n-o(n)$.

Theorem~\ref{thm:rectlinear_random_complex} is really a consequence of a technique due to Fox, Pach and Suk~\cite{approximating_crossing}, which they used to show that the graph parameter $\mathrm{cr}_{\mathbb R^2}$ is well-behaved with respect to the graph cut norm. This technique relies, in turn, on the semi-algebraic regularity lemma of the same authors~\cite{fox_semialgebraic}. Similar ideas were used by Gundert to study crossing numbers of expanding simplicial complexes in Chapter 8 of her PhD thesis~\cite{gundert2013expansion}.

\subsubsection{Antipodal drawings}

 There has been a lot of very interesting work confirming the Harary-Hill conjecture under additional hypotheses on the drawing. This line of research began with a result found independently in~\cite{AbregoFernandez‑Merchant2005}, where it was shown that $\Cr_{\mathbb R^2}(K_n)$ is at least the Harary-Hill bound, and in~\cite{LovaszVesztergombiWagnerWelzl2004}, where it was further proven that $\Cr_{\mathbb R^2}(K_n)$ is asymptotically larger than $\frac{n^4}{64}$. Building on these breakthroughs, the hypotheses on the drawing that imply the Harary-Hill conjecture have been relaxed in several directions~\cite{abrego2012two,abrego2014shellable,balko2015crossing}. We emphasize the case of \emph{bi-shellable} drawings~\cite{abrego2018bishellable}, which generalizes all the previously discovered ones, and a more recent remarkable result~\cite{streltsova2025spherical} which confirms the Harary-Hill conjecture for $\Cr_{\mathbb S^2}$. We prove yet another result of this kind for all dimensions.

Given a simplicial complex $X$ with a free $\mathbb Z_2$ action, we define the \emph{antipodal crossing number} of $X$ to be the smallest crossing number $\Cr(f)$, where the map $f \colon X \to \mathbb S^{2d}$ is required to be piece-wise linear, generic and $\mathbb Z_2$-equivariant. We denote this crossing number by $\Cr^{\mathbb Z_2}(X)$.

Let $K_1^{d+1}$ be the unit ball of the Minkowski space $\ell_1^{d+1}$; this polytope is known as the \textit{cross polytope}, and it is the convex hull of the vectors in the standard basis and their antipodals, that is, $K_1^{d+1}=\conv(e_1,-e_1,e_2,-e_2, \ldots e_d,-e_d,e_{d+1},-e_{d+1})$. 
Its boundary is a $d$-dimensional topological sphere with a triangulation (which possesses many symmetries) $\partial(K_1^{d+1})$ that can be equivalently constructed as the $(d+1)$-fold join of $\mathbb S^0$, denoted by $(\mathbb S^0)^{*(d+1)}$. 
See~\cite{Matousek2003} for more details.

 Consider the simplicial complex $(\mathbb S^0)^{*n}_{(d)}$, a subcomplex of the simplex on $2n$ vertices, $\Delta^{2n-1}_{(d)}$, in which we erase the edges of a perfect matching, together with all the faces that contain at least one edge of this matching. It is easy to see that $(\mathbb S^0)^{*n}_{(d)}$ contains a $1-o(1)$ fraction of the $d$-faces of $\Delta^{2n-1}_{(d)}$.
 Each factor $\mathbb S^0$ comes with a free $\mathbb Z_2$ action that interchanges the vertices. 
 We extend this action linearly on faces to all of $(\mathbb S^0)^{(*n)}$, which we identify with the boundary of the $n$-dimensional cross polytope. 
 This is a triangulation of $\mathbb S^{n-1}$ and the involution is the antipodal map. 
 
 The situation is similar in the case of the $(d+1)$-partite complex. Namely, any perfect matching of the vertices defines a $\mathbb Z_2$ action. 
 We denote by $B^{\mathbb Z_2}(\Delta_{(d)})$ and $B^{\mathbb Z_2}([\infty]^{*(d+1)})$ the corresponding limiting crossing constants when restricting to antipodal drawings.

\begin{theorem}\label{antipodal}
We have that \[\Cr^{\mathbb Z_2}((\mathbb S^0)^{*n}_{(d)})={\binom{n}{d+1}}{\binom{n-(d+1)} {d+1}} ,\] and hence 
\[B^{\mathbb Z_2}(\Delta_{(d)})=   \frac{1}{2^{2d+2}}\frac{(d+1)^{d}}{d!}.\]
Also, 
\[\Cr^{\mathbb Z_2}([2n]^{*(d+1)})=n^{d+1} (n-1)^{d+1},\]
and thus \[B^{\mathbb Z_2}([\infty]^{*(d+1)})= \big (\frac{d+1}{4}\big)^{d+1}.\]
\end{theorem}

The proof of Theorem~\ref{antipodal} can be easily modified to yield the following generalization.

\begin{theorem}\label{thm:dense_Z2}
Let $1-2^{-d}\leq\alpha\leq1$. If $\mathcal X=\{X_n\}_{n=1}^{\infty}$ is a sequence of $d$-complexes with a free $\mathbb Z_2$ action, $2n$ vertices and $(\alpha+o(1))2^{d+1}\binom{n}{d+1}$ $d$-faces, then, whenever the crossing constant exists, we have
\[B^{\mathbb Z_2}(\mathcal X)\geq\frac{1}{2^{2d+4}}\frac{(\alpha2^{d+1}-(2^{d+1}-2))^2(d+1)^d}{\alpha d!}.\]
\end{theorem}

It is easy to see that Theorem~\ref{antipodal} cannot be improved: any antipodal drawing of $(\mathbb S^0)^{*n} _{(d)}$ or of $[2n]^{*(d+1)}$ in which every face is the spherical hull of its endpoints provides an upper bound that matches the lower bound.

To go from the first lower bound of Theorem~\ref{antipodal} to a result about $\Delta^{2n-1}_{(d)}$, we define an antipodal drawing of $\Delta^{2n-1}_{(d)}$ to be any extension of a $\mathbb Z_2$-equivariant map of $\mathbb (S^0)^{*n} _{(d)}\to \mathbb S^{2d}$ to a map $\Delta^{2n-1}_{(d)} \to \mathbb S^{2d}$. Here, we arbitrarily identify the vertices of $(\mathbb S^0)^{*n}$ with those of $\Delta^{2n-1}_{(d)}$. 

Let $\mu$ be a probability distribution on $\mathbb S^{2d}$. Let $\nu_i$ be a family of probability distributions on the space of piecewise linear functions from the $i$-dimensional simplex $\Delta^i$ to $\mathbb S^{2d}$ equipped with the topology of pointwise convergence. More precisely, we are going to define a family of distributions on $PL$ maps parametrized by $PL$ maps of the boundary. For each PL-map $f \colon \partial \Delta^i \to \mathbb S^{2d}$, we let $\nu_i(f)$ be a probability measure on PL-maps $g \colon \Delta^i \to \mathbb S^{2d}$, such that $g|_{\partial\Delta^i}=f$. Put $\nu=(\nu_1,\nu_2, \ldots \nu_d)$ and define a random $(\mu,\nu)$-drawing to be one in which $n$ points are chosen independently according to $\mu$ and, for every pair of points that were chosen as vertices $x,y$, we choose a random arc according to $\nu_1(\{x,y\})$. Then, inductively, each triangle whose edges are represented by arcs $\gamma_1,\gamma_2,\gamma_3$ is filled by a random independent map from $\Delta^2 \to \mathbb S^{2d}$, distributed according to $\nu_2(\gamma_1\cup \gamma_2\cup \gamma_3)$, and so on (at each step, we increase by one the dimension of the faces being created).

This provides a very general model for random drawings of the $d$-skeleton of a higher dimensional simplex. We say that $\mu$ is antipodal if the measures of any two sets of the form $A$ and $-A$ are identical. We say that a parametrized family of distributions $\nu$ is antipodal if, for every map $f \colon \partial \Delta^i \to \mathbb S^{2d}$, the random filling given by $\nu_i(f(\partial \Delta^i))$ is identical to the negative of the random filling given by $\nu_i(-f(\partial \Delta^i))$.

\begin{corollary} If
$\mu$ is an antipodal distribution and $\nu=(\nu_1,\nu_2,\nu_3, \ldots \nu_d)$ is a parametrized family of antipodal distributions, then \[\mathbb E_{f \sim (\mu,\nu)} \Cr(f(\Delta_{(d)}^{2n-1}))\geq  \frac{1}{2^{2d+2}}{2\binom{n}{d+1}}{\binom{2n-(d+1)}{d+1}}.\]
\end{corollary}

This corollary is sharp, as can be seen from the uniform distribution with geodesic faces.  

\begin{corollary}\label{hill_conjecture}The Harary-Hill conjecture is true for the class of antipodal drawings, i.e., for $d=1$:
\[\Cr^{\mathbb Z_2}(\Delta^{m-1}_{(1)})=\frac{1}{4} \left\lfloor \frac{m}{2} \right\rfloor \left\lfloor \frac{m - 1}{2} \right\rfloor \left\lfloor \frac{m - 2}{2} \right\rfloor\ \left\lfloor \frac{m - 3}{2} \right\rfloor.\]
\end{corollary}

The fact that there exist antipodal drawings that realize the Harary-Hill bound was discovered not long ago by Mohar (see~\cite{bojan2020_hill_conjecture, bonamy_mohar_wesolek2020,mohar_art_circuit_design}); the corresponding drawings are antipodal and geodesic on the round sphere. As mentioned above, this is easy to see for $\mathbb (S^0)^{*n}_{(1)}$. If we add an arc for each antipodal pair, we obtain $n {\binom{n-1}2}$ crossings between these arcs and the drawing of $(\mathbb S^0)^{*n}_{(1)}$, since every arc between two antipodal points must intersect every antipodal drawing of a vertex-disjoint $C_4$. Again this is best possible for an antipodal geodesic drawing. In total, this matches exactly the Harary-Hill bound, so to finish the construction we need to ensure that we can draw an arbitrary number of pairwise disjoint antipodal arcs. For this last requirement, we use a judicious choice of the positions of the points. Figure 4 shows a drawing of three such antipodal arcs, it is easy to see that this construction generalizes to configurations of any number of arcs between pairs of antipodal points, and we will provide a rigorous description of how to construct such a drawing. Finally, if we want a drawing of a complete graph with $2n+1$ vertices, we can take one with $2n+2$ and erase an arbitrary point. 

\begin{figure}[h]
\caption{Three arcs between antipodal points without a crossing}
\centering
\includegraphics[scale=0.5]{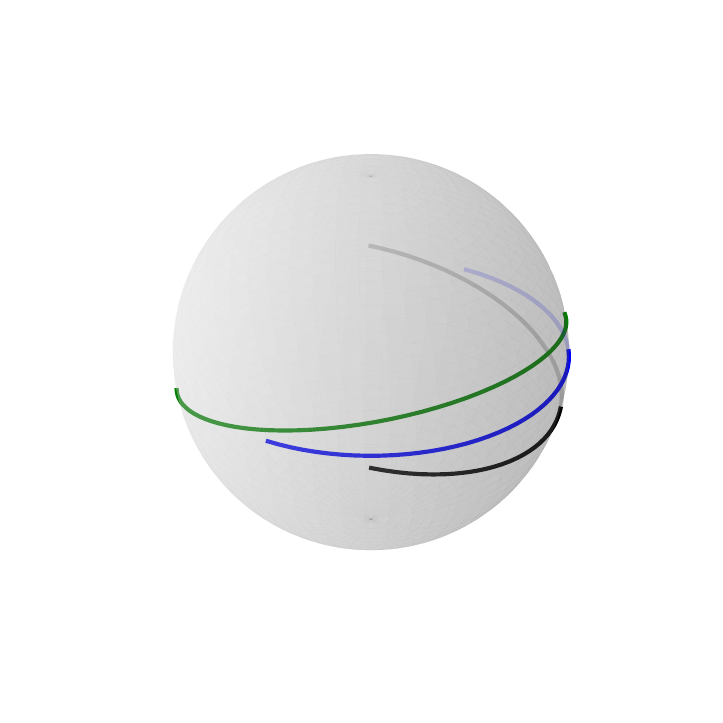}
\end{figure}

\begin{conjecture}\label{higher_hill} For any two positive integers $m$ and $d$,
\[\Cr^{\mathbb Z_2}(\Delta_{(d)}^{m-1})=\Cr(\Delta_{(d)}^{m-1}).\]
\end{conjecture}

By Theorem~\ref{antipodal}, Corollary~\ref{hill_conjecture}  and the constructions in~\cite{bojan2020_hill_conjecture} that we just explained, the case $d=1$ of this conjecture is equivalent to the Harary-Hill conjecture.

\subsection{The midrange and random geometric complexes.}

For technical reasons we will work with the independent pair-crossing number $\Cr_\iota(X)$.

Let $X_{s,t}$ be a $d$-dimensional simplicial complex that minimizes the independent pair-crossing number among complexes with $s$ faces of dimension $d-1$ and at least $t$ faces of dimension $d$ (this notation is somewhat ambiguous, as the complex $X_{s,t}$ depends on $d$; some other ambiguities of this kind will occur later on, but they should cause no confusion). It can be assumed, without loss of generality, that $X_{s,t}$ has precisely $\lceil t\rceil$ faces of dimension $d$. Define \[\kappa_d(s,t):=\Cr_\iota(X_{s,t}).\]
 Since a copy of $X_{s_1,t_1}$ can be drawn optimally without intersecting a copy of $X_{s_2,t_2}$, the function $\kappa_d$ satisfies
$\kappa_d(s_1+s_2,t_1+t_2)\leq \kappa_d(s_1,t_1)+\kappa_d(s_2,t_2).$
In particular, for fixed $a$, $\kappa_d(s,a s)$ is sub-additive on $s$ so, by Fekete's subadditivity lemma~\cite{Fekete1923}, for every $a$ the following limit exists and is finite \[b_d(a):=\lim_{s \to \infty} \kappa_d(s,a s) \frac{s^{d+1}}{(a s)^{d+2}}=\frac{1}{a^{d+2}} \lim_{s \to \infty}\frac{\kappa_d(s,a s)}{s}.\]

\noindent\textbf{\textit{Remark.}} To lighten notation, $\kappa_1,b_1$ were defined using the crossing number, and $\kappa_d, b_d$ (for $d>1$) are being defined using the independent pair-crossing number. A related function $\kappa_d'$ will also be defined below using the independent pair-crossing.

From Corollary~\ref{crossing_lemma}, it follows that for $d\neq 2 $ we have $b_d(a)=0$ if $a \in [0,d+2]$, $b_d(a)= \frac{a-(d+2)}{a^{d+2}}$ if $a\in[d+2,d+2+\frac{2}{d+1}]$, $b_d(a)\ge \frac{a-(d+2)}{a^{d+2}}$ if $a\in[d+2+\frac{2}{d+1},d+3+\frac{1}{d+1}]$, and $b_d(a)\geq \frac{(d+1)^{d+1}}{(d+2)^{2d+3}}$ if $a\geq d+3+\frac{1}{d+1}$.

Define  \[B_d:=\limsup_{a \to \infty} b_d(a).\]
The best lower bound we know of for $B_d$, $d\geq 3$, comes from the above and reads as \[B_d\geq \frac{(d+1)^{d+1}}{(d+2)^{2d+3}}.\]

If we consider the disjoint union of $k$ copies of the $d$-skeleton of the $(n-1)$-simplex, which has $s=k{\binom nd} $, $t=k{\binom{n}{d+1}}$, and independent pair crossing number $k\Cr_\iota(\Delta_{(d)}^{n-1})$, then from Theorem~\ref{Moon_Wendel} we conclude that \[B_d\leq B^\iota(\Delta_{(d)})\leq \frac{1}{2^{2d+2}}\frac{(d+1)^{d}}{d!}.\] 

More generally, the same observation yields the following:

\begin{lemma}
\label{lem:dense_seeds}
Let $d\geq 2$ and let
$\mathcal X=\{X_n\}_{n=1}^\infty$ be a sequence of $d$-simplicial complexes, and put $s_n:=|X_n(d-1)|$, $t_n:=|X_n(d)|$, and $a_n:=\frac{t_n}{s_n}$. Assume that $a_n\to\infty$ and $\frac{a_{n+1}}{a_n}\to 1$, and that the pair crossing constant $B^\iota(\mathcal X)$ exists. Then, we have that 
\[B_d\leq B^\iota(\mathcal X)
\]
\end{lemma}

Notice that if 
$|X_n(d-1)|\sim c n^d$, $|X_n(d)|\sim c'n^{d+1}$, and $B^{\iota}(\mathcal X)$ exists, then the hypotheses of this lemma are satisfied. This is how this result will be used.

\begin{lemma}\label{midrange_sampling} Let $d\geq2$. For any $\varepsilon>0$ and any $C>0$, there exists a positive number $N=N(\varepsilon,C)$ such that  for any $d$-dimensional simplicial complex $X$ satisfying 
$|X(d-1)|\geq N$, $\frac{|X(d)|}{|X(d-1)|}\geq N$ and 
\[\frac{|X(d-i)|}{|X(d-1)|}\leq C\left(\frac{|X(d-1)|}{|X(d)|}\right)^{i-1}\] for all $2\leq i\leq d$, one has \[\Cr_\iota(X)\frac{|X(d-1)|^{(d+1)}}{|X(d)|^{(d+2)}}> (1-\varepsilon)  B_d.\]
\end{lemma}

We also need a variant of $\kappa_d$. We say that a simplicial complex is $q$\textit{-equitable} if, for any vertex $v$, \[\deg_d(v)\leq \left ( \frac{(d+1)|X(d)|}{|X(d-1)|}\right )^{q}.\] 
Denote the set of $q$-equitable complexes by $\mathcal E_q$.
For integers $s$ and $t$, define \[\kappa'_d(s,t):=\min_{\begin{subarray}{l} X \in \mathcal{E}_q , |X(d-1)|=s,\,|X(d)|= t \\ \frac{|X(i-1)|}{|X(i)|}<C\frac{s}{t}\  \forall i \in\{0,1, \ldots d-1\}  \end{subarray}} \Cr_\iota(X),\] where we use the convention that $|X(-1)|=1$.

Notice the abuse of notation: the function $\kappa'_d$ depends on $q$ and $C$. In the rest of the paper we assume that $q$ is a constant that depends on $d$ only and satisfies $d < q<d+1$, and that $C$ is a fixed arbitrary positive number.

Since the disjoint union of two $q$-equitable complexes with the same ratio $t/s$ is $q$-equitable, and the condition involving $|X(i-1)|$ and $|X(i)|$ is also closed under disjoint unions, as for $\kappa_d$, $\kappa'_d(s,as)$ is a subadditive function of $s$. Hence, the following limit is well defined for rational $a$ along sequences for which $as\in\mathbb Z$:
\[b'_d(a):=\lim_{s \to \infty} \kappa'_d(s,a s) \frac{s^{d+1}}{(a s)^{d+2}}=\frac{1}{a^{d+2}} \lim_{s \to \infty}\frac{\kappa'_d(s,a s)}{s}.\]
For $d\geq 2$, the right hand side above is in turn lower bounded by $ b_d(a)$. This implies that, again for rational $a$, \[B_d':=\limsup b'_d(a)\geq B_d\] for all $d\geq 2.$

\begin{lemma}\label{midrange_blowup}
Let $d\geq 2$. For any $\varepsilon>0$, there exists $N=N(\varepsilon)$ such that if $s,t/s,s^{d+1}/t^d>N$, then $\kappa'_d(s,t)\frac{s^{d+1}}{t^{d+2}}< B_d'(1+\varepsilon)$.
\end{lemma}

Putting together these two lemmas we obtain a weak version of Theorem~\ref{midrange}. 

\begin{theorem} Let $d\geq2$. For any $s,t$ such that $s\ll t \ll s^{1+1/d}$, we have
\[B_d\leq\liminf_{s\to\infty}\kappa'_d(s,t)\frac{s^{d+1}}{t^{d+2}}\leq\limsup_{s\to\infty}\kappa'_d(s,t)\frac{s^{d+1}}{t^{d+2}}\leq B_d'.
\]
\end{theorem}

 For $d=1$, it is not hard to show that $B_1=B_1'$ (see~\cite{pach1996new}).  

 \begin{question} Is $B_d'=B_d$?
\end{question}

 Trying to generalize the argument of~\cite{pach1996new} that implies a positive answer to the previous question for $d=1$, one runs into the following type of problem, which seems interesting on its own right. For this question we say that an edge coloring of a graph with $e$ edges is balanced if the number of edges of each color is $\Omega(e/k)$.

\begin{question}
Is there a spatial graph (a graph embedded in $\mathbb R^3$) and an integer $k>1$ such that for any balanced $k$-coloring of the edges there exist two monochromatic cycles of different colors that are linked?
\end{question}

This statement is written in a way that is reminiscent of the famous Conway-Gordon-Sachs Theorem. To show that $B_2=B_2'$, we would need to give a negative answer to this question in the case of rectilinear spatial embeddings. 
We remark that one can show that, for any $k>1$, there exists a balanced $k$ coloring of the edges of the complete graph drawn rectilinearly with the vertices on the moment curve $\{(t,t^2,t^3)\in \mathbb R^3: t \in \mathbb R\}$, such that no two monochromatic cycles of different colors are linked.

Define
\[
\kappa_{d,\mathbb S}(s,t):=\min_{|X(d-1)|=s,\ |X(d)|\geq t}\Cr_{\mathbb S}(X).
\]
Unlike the ordinary crossing number, the spherical crossing number is not obviously subadditive under disjoint unions: two geodesically spherical drawings cannot in general be placed in disjoint balls while keeping every face geodesic. We therefore define, for $a>0$,
\[
b_{d,\mathbb S}(a):=\liminf_{s\to\infty}\kappa_{d,\mathbb S}\bigl(s,\lfloor as\rfloor\bigr)\frac{s^{d+1}}{\lfloor as\rfloor^{d+2}},
\]
and define the \emph{spherical midrange crossing constant} as
\[
B_{d,\mathbb S}:=\limsup_{a\to\infty} b_{d,\mathbb S}(a).
\]

Since every geodesically spherical drawing is, after deleting one point of the sphere and applying a homeomorphism to $\mathbb R^{2d}$, a topological drawing, we have
\[
\Cr_\iota(X)\leq \Cr_{\mathbb S^{2d}}(X)
\]
for every $X$. Consequently,
\[
\kappa_d(s,t)\leq \kappa_{d,\mathbb S}(s,t),
\]
and hence $B_d\leq B_{d,\mathbb S}$.

\begin{theorem} We have that \[B_{2,\mathbb S}\leq \frac{3\cdot 53^2\cdot 2^5}{5^4(8\pi-9\sqrt{3})^4}\sim 0.0520.\]
\end{theorem}

This upper bound on $B_2$ arises from the $2$-skeleton of the flag complex (also known as clique complex) of a random geometric graph in $\mathbb S^4$. That is, we consider a geometric random graph on $\mathbb S^4$ and add a $2$-face for every triangle in the graph.
This is a generalization of the aforementioned model from which the bound $B_1\leq \frac{8}{9\pi^2}$ is derived. The key to the proof is to relate the crossing number to the area of random triangles using the Crofton formula and the fact that this random simplicial complex is invariant under sphere isometries.

The analogous calculations to obtain an upper bound on $B_d$ for $d>2$ are difficult, but for any $d \geq 4$, the bounds provided by the $d$-skeleton of the flag complex of the geometric random graph with $\alpha \to 0$ seem to be greater than $B(\Delta_{(d)})$. We don't have a clear candidate to upper bound $B_d$, but we prove the following exponential improvement over the Wendel-Moon bound.

\begin{theorem}
\label{thm:volume_threshold_exponential}
There is $d_0$ such that, for every $d\ge d_0$,
\[
B_d
\le
e^{-(d+1)/200}\, \frac{1}{2^{2d+2}}\frac{(d+1)^{d}}{d!}
\]
\end{theorem}

This result also comes from geometric random complexes, but instead of thresholding the length of the edges, we threshold the volume of $d$-faces. The proof is variational and combines an observation about Gaussian random matrices with Gale duality.

\section{The crossing lemma}
\subsection{The lower bound}

Let us spell out the proof of the crossing number inequality in higher dimensions, which follows the proof of the analogous result for graphs. For rectilinear crossing number, this argument goes back to~\cite{Dey1993}. We  clarify the importance of lemma~\ref{whitney_trick} in the piecewise linear setting. 

 We first derive a simple corollary of Theorem~\ref{Adi}.

\begin{corollary}\label{coro:cr-d+2}  For any simplicial complex $X$ of dimension $d$, $\Cr(X)\geq |X(d)|-(d+2)|X(d-1)|$. 
\end{corollary}
\begin{proof}

Consider a drawing $f:X\to\mathbb R^{2d}$ attaining $\operatorname{cr}(X)$. Let $Y\subset X$ be a simplicial complex which is maximal with the property that $\Cr(f(Y))=0$ (so that, in particular, $\Cr(Y)=0$). Since the map $f$ can be assumed to be generic, $Y_{(d-1)}=X_{(d-1)}$. 
By Theorem~\ref{Adi}, $|Y(d)|\leq (d+2)|Y(d-1)|$.  
Now, for each $d$-face $\sigma \in X(d)\setminus Y(d)$, let $Y'$ be the smallest subcomplex of $X$ that contains $Y$ and $\sigma$. By maximality of $Y$, $\Cr(f(Y'))>0$. Thus, the image of $\sigma$ under $f$ must intersect some other $d$-face of $Y$ in its relative interior. Repeating this for every such $\sigma$, each of these faces contributes at least one crossing, and hence \[\Cr(X)=\Cr(f(X))\geq |X(d)|-|Y(d)|\geq |X(d)|-(d+2)|X(d-1)|.\]
\end{proof}

 Wu~\cite{Wu1957} and Shapiro~\cite{Shapiro1957} completed Van Kampen's program~\cite{vanKampen1932}, which can be summarized as follows: the Van Kampen obstruction characterizes embeddability of $d$-dimensional simplicial complexes in $\mathbb R^{2d}$ for $d>2$. See~\cite{melikhov2009van} for much more on this topic and~\cite{freedman1994van} for a statement that is almost identical to the next lemma (that summarizes lemmas 4 and 5 there).

\begin{lemma}\label{whitney_trick}\cite{freedman1994van} Let $f\colon X \to \mathbb{R}^{2d}$ be a generic PL immersion of a simplicial complex $X$ of dimension $d\geq 3$.
Suppose that either for some $d$-simplex $\sigma$ $f(\sigma)$ is not an embedding, or $f(\sigma)$ intersects $f(\sigma')$ in its interior for two $d$-simplexes $\sigma$ and $\sigma'$ having a common vertex in $X$. Then, there is an immersion $g \colon X \to \mathbb R^{2d}$ differing from $f$ only in $\sigma$ (in the second case in $\sigma$ and $\sigma'$) such that the number of self-intersection points of $g(\sigma)$ (or, in the second case, between $g(\sigma)$ and $g(\sigma')$) is one less than the number of self-intersection points of $f(\sigma)$ (or, in the second case, intersections between $f(\sigma)$ and $f(\sigma')$), and the rest of the singularities are unchanged.
\end{lemma}

\noindent\textbf{\em{Remark:}} To prove Lemma~\ref{whitney_trick}, one isotopes the map along a $2$-disk that avoids the rest of the image of $f$; this can be assumed to exist by a general position-dimension counting argument, provided that the dimension of the codomain is at least $5$. This is a version of the so-called \emph{Whitney trick}. The difficulty in $\mathbb R^4$ is similar to why the Poincar\'e conjecture is easier to prove in dimensions $d>4$ than in dimensions $d=3,4$. The main result of~\cite{freedman1994van} implies that the analogue for $d=2$, $2d=4$, is false.

\begin{proof}[Proof of parts 1) and 3) of Corollary~\ref{crossing_lemma}] Let $X$ be an arbitrary $d$-dimensional simplicial complex and $f\colon X \to \mathbb S^{2d}$ be a map that minimizes crossings; $f$ can be assumed to be a generic map. Moreover, if $d\neq 2$, by Lemma~\ref{whitney_trick}, we can assume that there do not exist any two simplices $\sigma,\sigma'\in X(d)$ which share at least one vertex and such that the relative interior of $f(\sigma)$ intersects the relative interior of $f(\sigma')$. In the case of $d=2$ and facewise affine or spherically geodesic triangles, the previous property is immediately satisfied as soon as the vertices are in general position, which can be assumed for a map that minimizes the number of crossings. In all cases that we treat, it can be assumed that crossings only arise between vertex-disjoint simplices. This is precisely why the assumption that $d\neq 2$ is important.

Corollary~\ref{coro:cr-d+2} and Lemma~\ref{whitney_trick} imply that, for $d\neq 2$, $\Cr_\iota(X)\geq |X(d)|-(d+2)|X(d-1)|$.

Let $R_p$ be a random induced subcomplex where each vertex of $X$ is included with probability $p$ independently from the rest. 
Thus,
\[\Cr_\iota(R_p)\geq |R_p(d)|-(d+2)|R_p(d-1)|.\]
Using independence, the above inequality, and linearity of expectation,
\[p^{2d+2}\Cr_\iota(X)\geq \mathbb E[\Cr_\iota(R_p)]\geq \mathbb E\left[|R_p(d)|-(d+2)|R_p(d-1)|\right]=p^{d+1} |X(d)|-(d+2) p^d |X(d-1)|.\]
Hence 
\[\Cr(X)\geq \Cr_\iota(X)\geq  \frac{|X(d)|}{p^{d+1}}-\frac{(d+2)|X(d-1)|}{p^{d+2}}.\]
Let us set $p=\frac{a|X(d-1)|}{|X(d)|}$ (the assumption $|X(d)|\ge a|X(d-1)|$ implies that $p\leq 1$).
Replacing $p$ in the previous formula and simplifying, we obtain the claimed bound. 
\end{proof}
Notice that if it were not for Lemma~\ref{whitney_trick}, the expected crossing number could be greater than $p^{2d+2} \Cr(X)$: self-crossings 
and crossings between incident faces survive with larger probabilities.
Lemma~\ref{whitney_trick} allows us to use $\Cr_{\iota}$ instead of $\Cr$ in the linear inequality  $\Cr(X)\geq |X(d)|-(d+2)|X(d-1)|$. See the introduction of~\cite{schaefersurvey} for a discussion of this issue in the case of $d=1$ for other variants of the crossing number of graphs. 

The above proof actually yields the following more general lemma:

\begin{lemma}\label{amplification}
If a linear inequality of the form
$\Cr_{\iota}(X)\geq \alpha |X(d)|-\beta|X(d-1)|$ holds for all complexes (where $\alpha$ and $\beta$ are positive numbers), then under hypothesis $|X(d)|\geq a|X(d-1)|$, the following crossing number inequality is valid:
\[\Cr(X)\geq \Cr_{\iota}(X)\geq \frac{\alpha a-\beta}{a^{d+2}}\frac{|X(d)|^{d+2}}{|X(d-1)|^{d+1}}.\]
\end{lemma}
As elsewhere, in this lemma we can replace $\Cr$ with $\Cr_{\mathbb S^{2d}}$ or $\Cr_{\operatorname{Aff}}$.  We remark that, in the case that $|X(d)|= a|X(d-1)|$, this simplifies to the linear inequality that we started with: $\Cr(X)\geq (\alpha a-\beta)|X(d-1)|=\alpha|X(d)|-\beta|X(d-1)|$. 

Now we show the crossing lemma for $d=2$, as stated  in Proposition~\ref{cased=2}. We are going to use the following observation.

\begin{lemma}If $X$ is a $d$-complex that embeds in $\mathbb S^{2d}$, then $X *[q]$ is a $(d+1)$-complex that embeds in $\mathbb S^{2d+2}$, moreover $\Cr_\iota(X)=0$ implies $\Cr_\iota(X*[q])=0$.
\end{lemma}

\begin{proof}
View $\mathbb S^{2d+2}$ as the unit sphere in $\mathbb R^{2d+1}\oplus\mathbb R^2$. It contains the equatorial sphere $\mathbb S^{2d}\times\{0\}$ and the orthogonal circle $\{0\}\times\mathbb S^1$. Given $f\colon X\to\mathbb S^{2d}$, place its image in the first sphere and choose distinct points $y_1,\ldots,y_q$ on the orthogonal circle.
Starting from $f$, we construct $F\colon X*[q]\to\mathbb S^{2d+2}$ by mapping each cone $X*\{y_i\}$ to the spherical cone with base $f(X)$ and apex $y_i$: each segment from $y_i$ to $x$ is mapped to the shorter great-circle arc from $y_i$ to $f(x)$. Explicitly, writing $y_i$ for its $\mathbb R^2$ coordinate,
\[F([x,y_i,t])=\bigl(\sin(\pi t/2)f(x),\cos(\pi t/2)y_i\bigr),\qquad 0\leq t\leq1,\]
where $t=0$ is the apex and $t=1$ is the base. The formula is well defined under the join identifications and takes values in $\mathbb S^{2d+2}$. Distinct cones meet only along their common base, and within each cone the formula is injective whenever $f$ is injective. Thus $F$ is an embedding if $f$ is.

Now suppose that $\Cr_\iota(X)=0$, and choose $f$ with no independent crossings. Two vertex-disjoint $(d+1)$-faces of $X*[q]$ have the form $\sigma*\{y_i\}$ and $\tau*\{y_j\}$, where $i\neq j$ and $\sigma,\tau\in X(d)$ are vertex-disjoint. Their images can meet only in the common base, where their intersection is $f(\sigma)\cap f(\tau)=\varnothing$. A sufficiently close generic PL approximation of $F$ preserves these finitely many disjointness relations, since the relevant images are compact. Thus, $\Cr_\iota(X*[q])=0$.

\end{proof}

\begin{proof}[Proof of Proposition~\ref{cased=2}]Let $X$ be a $2$-complex with $\Cr_\iota(X)=0$. Consider $Y=X*[q]$; it has $|Y(3)|=q|X(2)|$, $|Y(2)|=|X(2)|+q|X(1)|$, and $\Cr_\iota(Y)=0$. Hence, by Adiprasito's theorem for $d=3$, combined with Lemma~\ref{whitney_trick}, $|Y(3)|\leq 5|Y(2)|$. This yields $q|X(2)|\leq 5(|X(2)|+q|X(1)|)$ or, in other words, $|X(2)|\leq 5 \frac{q}{q-5}|X(1)|$. Since this is true for all $q$, $|X(2)|\leq 5|X(1)|$. The contrapositive is that if $|X(2)|> 5|X(1)|$, then $\Cr_{\iota}(X)>0$. Arguing as in Corollary~\ref{coro:cr-d+2} we derive $\Cr_\iota(X)\geq |X(2)|-5|X(1)|$. Lemma~\ref{amplification} finishes the proof.  
\end{proof}

\subsection{Very sparse complexes with few crossings}
\label{subsec:sparse-complexes}

We now give a construction showing that the linear crossing inequality is sharp on the first sparse interval.

\begin{proposition}\label{prop:first-sparse-interval}
For every $d\geq1$ and $d+2\leq a\leq d+2+\frac{2}{d+1}$, we have $b_d(a)\leq\frac{a-d-2}{a^{d+2}}$. The construction gives $\mathrm{PL}$ drawings in which every $d$-face is crossed at most once.
\end{proposition}

\begin{proof}
Let $n>2d+3$ be relatively prime to $d+1$, and label the vertices of the cyclic polytope $P=C(n,2d)$ by $0,1,\ldots,n-1$ in their moment-curve order (see~\cite[Chapter 0]{Gunter_polytopes_lectures} for background on cyclic polytopes). Call a $(d+1)$-element vertex set separated if it contains no cyclically consecutive vertices, including the pair $\{n-1,0\}$. By Gale's evenness criterion, every set of at most $d$ vertices is a boundary face, and a $(d+1)$-element set spans a boundary $d$-face exactly when it is not separated. Also, the unique affine dependence among $2d+2$ moment-curve vertices has alternating signs. Thus, two vertex-disjoint $d$-simplices cross exactly when their vertices alternate in cyclic order, in which case their relative interiors meet transversely in one point. Simplices sharing vertices meet only in their common face, since any $2d+1$ vertices are affinely independent.

For $r\in\mathbb Z/n\mathbb Z$, let $\mathcal A_r$ consist of the separated sets whose labels sum to $r$ modulo $n$. There are $\frac{n}{d+1}\binom{n-d-2}{d}$ separated sets: mark one vertex and count the $d+1$ cyclic gaps, which are at least two and sum to $n$, then divide by $d+1$. Cyclic translation $A\mapsto A+1$ maps $\mathcal A_r$ bijectively to $\mathcal A_{r+d+1}$. Our choice of $n$ therefore gives $|\mathcal A_r|=\frac{1}{d+1}\binom{n-d-2}{d}$ for every $r$.

Let $A$ and $B$ be $(d+1)$-element sets and suppose that they alternate in cyclic order. Choose integer lifts $a_0<b_0<a_1<b_1<\cdots<a_d<b_d<a_0+n$ of their $\mathbb Z/n\mathbb Z$-valued labels. Summing the inequalities $1\leq b_i-a_i\leq a_{i+1}-a_i-1$, where $a_{d+1}=a_0+n$, gives
\[d+1\leq\sum_{i=0}^d(b_i-a_i)\leq n-d-1.\]
The sum is congruent modulo $n$ to the difference of the residue classes of $B$ and $A$. Consequently, among the simplices in $\mathcal A_r\cup\cdots\cup\mathcal A_{r+d+1}$, where $0\leq r\leq n-d-2$, crossings occur only between the first and last classes. Writing $A$ for the set in the first class, equality holds on the left, so $B=A+1$. Conversely, every separated $A$ alternates with $A+1$. Thus the crossing pairs in this block are precisely $\{A,A+1\}$ with $A\in\mathcal A_r$.

Glue two copies of $P$ along their boundary to obtain a PL $2d$-sphere, and draw the whole $d$-skeleton of $\partial P$ on their common boundary. In the first copy, we draw the affine simplices from $\mathcal A_0\cup\cdots\cup\mathcal A_d$, and in the second we draw those from $\mathcal A_{d+2}\cup\cdots\cup\mathcal A_{2d+2}$. The proper faces of all these simplices already lie on the common boundary of the two copies of $P$, their relative interiors lie in their assigned open half, and the preceding observation shows that this drawing has no crossings. We may now add any selection of the simplices of $\mathcal A_{d+1}$ in the first copy and of $\mathcal A_{2d+3}$ in the second. Each added simplex crosses exactly one previously drawn simplex, its translate by $-1$, and no two added simplices cross. All these vertex sets are distinct. Removing a point outside the drawing gives the required PL drawing in $\mathbb R^{2d}$.

To obtain asymptotic density $a$, add exactly $\left\lfloor(a-d-2)\binom{n-d-2}{d}\right\rfloor$ of these simplices, which is possible throughout the stated interval, and call the resulting complex $X_n$. Since the nonseparated sets are precisely the boundary $d$-faces, the number of such faces is
\[f_d(\partial P)=\binom{n}{d+1}-\frac{n}{d+1}\binom{n-d-2}{d}=d\binom nd+O_d(n^{d-1}).\]
The initial drawing (i.e., the drawing without any elements of $\mathcal A_{d+1}$ or $\mathcal A_{2d+3}$) contains $2\binom{n-d-2}{d}$ further $d$-faces. It follows that $|X_n(d-1)|=\binom nd$, $|X_n(d)|=a\binom nd+O_d(n^{d-1})$, and $\Cr(X_n)\leq(a-d-2)\binom nd+O_d(n^{d-1})$.

Finally, the disjoint-union argument makes $a\mapsto a^{d+2}b_d(a)$ finite and convex, hence continuous on $(0,\infty)$. Taking disjoint copies of $X_n$ bounds its value at $|X_n(d)|/|X_n(d-1)|$ by $\Cr(X_n)/|X_n(d-1)|$. Letting $n$ tend to infinity proves the claim.
\end{proof}

For $d\neq2$, combining Proposition~\ref{prop:first-sparse-interval} with Corollary~\ref{coro:cr-d+2} gives
\[
b_d(a)=\frac{a-d-2}{a^{d+2}}\qquad\text{for }d+2\leq a\leq d+2+\frac{2}{d+1}.
\] This proves part 2) of Corollary~\ref{crossing_lemma}.

\subsection{Very dense complexes}

Our goal in this section is to prove Theorem~\ref{Moon_Wendel} and Corollary~\ref{double_count}. We begin by showing that the constant $B^\iota(\Delta_{(d)})$ exists. The proof of this fact is a straightforward generalization of the one for $d=1$ which can be found in~\cite{RichterThomassen1997}.

\begin{proof}[Proof that the constant $B^\iota(\Delta_{(d)})$ exists.]
For every $n$, let $f_n$ be a drawing of $\Delta_{(d)}^{n-1}$ which attains $\Cr_\iota(\Delta_{(d)}^{n-1})$ and write $p_n:=\mathbb P(f_n(\sigma)\cap f_n(\sigma')\neq \emptyset)$, where $\sigma,\sigma'$ is a random uniform vertex-disjoint pair of 
 $d$-faces in $\Delta_{(d)}^{n-1}$. Notice that \[p_{n+1}=\frac{1}{n+1}\sum_v \mathbb P(f_{n+1}(\sigma)\cap f_{n+1}(\sigma')\neq \emptyset| \sigma \cap  \sigma'=\emptyset, v \notin \sigma \cup \sigma')\geq p_n.\] If we could choose a drawing $f_{n+1}$ of $\Delta_{(d)}^{n}$ such that, for every vertex $v$, $f_{n+1}$ restricts to a drawing of $\Delta_{(d)}^{n}\setminus \{v\}$ that realizes $\Cr_\iota(\Delta_{(d)}^{n-1})$, then we would have equality. 

Clearly, $p_n\leq 1$ for every $n$, so the sequence $\{p_n\}_{n=1}^\infty$ converges to a non-negative constant. On the other hand, \[\Cr_\iota(\Delta^{n-1}_{(d)}) \frac{{\binom{n}{d}}^{(d+1)}}{{\binom{n}{d+1}}^{d+2}}=p_n \frac{1}{2}\frac{(d+1)^d}{d!}+o(1),\]
so $B^\iota(\Delta_{(d)})$ exists.
\end{proof}

The same proof shows that $B^\mathbb S(\Delta_{(d)})$, $B^\mathbb R(\Delta_{(d)})$ exist. The proof is also valid if we replace $\Delta_{(d)}$ with $[\infty]^{*(d+1)}$.

   For $\Cr$ we normalize the number of crossings instead: put $c_n:=\Cr(\Delta^{n-1}_{(d)})/N_n$, where $N_n=\frac12\binom{n}{d+1}\binom{n-d-1}{d+1}$. If all crossings of an optimal drawing are between vertex-disjoint faces, every crossing survives in exactly $n-2d-1$ of the $n+1$ vertex deletions, so $c_n\leq c_{n+1}$ (but a pair of disjoint faces can intersect multiple times). This holds for $d\geq3$ by Lemma~\ref{whitney_trick} and for $d=1$ by elementary redrawing, and $c_n\leq 2^{-2d-1}$ by the upper bound below. Hence $B(\Delta_{(d)})$ exists for $d\neq2$.
   
\begin{proof}[Proof of the bounds in Theorem~\ref{Moon_Wendel}]

The Van-Kampen theorem tells us that $\Cr_\iota(\Delta_{(d)}^{2d+2})=1$. A double counting argument gives
\[\Cr_\iota(\Delta_{(d)}^{n-1})\geq \frac{{\binom{n}{2d+3}}}{n-2d-2}= \frac{n^{2d+2}}{(2d+3)!}-O_d(n^{2d+1}),\] which implies that \[B^\iota(\Delta_{(d)})\geq \frac{(d+1)^{(d+2)}d!}{(2d+3)!}).\]

For the upper bound we use the probabilistic method. We show that for a certain random drawing 
$\phi \colon\Delta^{n-1}_{(d)} \to \mathbb S^{2d}$, the expected crossing number satisfies \[\mathbb E[\Cr_\iota(\phi)]={\binom{n}{d+1}}{\binom{n-(d+1)}{d+1}}\frac{1}{2^{2d+2}}.\]
This will conclude the proof, as there exists some deterministic map $\phi_0$ for which 
$\Cr_\iota(\phi_0) \leq \mathbb E[\Cr(\phi)]$, and
\[\frac{|\Delta^{n-1}_{(d)}(d-1)|^{(d+1)}}{|\Delta^{n-1}_{(d)}(d)|^{d+2}}={\binom{n}{d}}^{d+1}{\binom{n}{d+1}}^{-d-2}\sim_{n\to+\infty}\frac{(d+1)^{(d+1)}(d+1)!}{n^{2d+2}}.\]
Approximating terms of the form ${\binom nk}$ for $k$ fixed and $n$ large by $\frac{n^k}{k!}$ and then doing some simple calculations yields the claimed bound on the crossing constant.

The random drawing $\phi$ is defined as follows. The points representing the vertices are chosen independently at random from the uniform distribution on $\mathbb S^{2d}$ (in fact, any centrally symmetric distribution which is absolutely continuous with respect to the surface measure gives the same bound by the same argument) and, for every $d+1$ points $v_1,v_2, \ldots ,v_{d+1}$ that have been chosen, the simplex is drawn as the geodesic spherical simplex \[\sph(v_1,v_2\ldots v_{d+1}):=\cone(v_1,v_2, \ldots v_{d+1})\cap \mathbb S^{2d},\] where $\cone(u_1,\dots,u_k)$ is the positive cone $\cone(u_1,\dots,u_k):=\{\sum_{i=1}^{k} \lambda_i u_i: \lambda_i\geq 0\ \forall\ i\}$. It is enough to show that the probability that two random simplices intersect is $\frac{1}{2^{2d+1}}$. 

To see this, let $v_1,v_2, \ldots v_{d+1},u_1,u_2,\ldots u_{d+1}$ be chosen independently at random from the uniform measure on $\mathbb S^{2d}$. 
Put $V:=\sph(v_1,v_2, \ldots v_{d+1})$ and $U:=\sph(u_1,u_2, \ldots u_{d+1})$. 
Let \[a_1,a_2, \ldots a_{d+1}, b_1,b_2, \ldots ,b_{d+1}\] be arbitrary generic points on $\mathbb S^{2d}$. 
For each $i$, condition $v_i$ to be either $a_i$ or $-a_i$ with equal probability and $u_i$ to be either $b_i$ or $-b_i$ with equal probability. The set of spherical $d$-simplices with vertices in the set \[\{a_1,-a_1,a_2,-a_2, \ldots ,a_{d+1},-a_{d+1}\}\] 
spans an equatorial $d$-dimensional sphere that is triangulated with $2^{d+1}$ simplices. 
In fact, this triangulation is isomorphic to the $(d+1)$-fold join $(\mathbb S^0)^{*(d+1)}$.
A similar argument applies to the $(d+1)$-tuples of points in $\{b_1,-b_1,b_2,-b_2, \ldots b_{d+1},-b_{d+1}\}$. Here, by equatorial sphere we mean the intersection of the standard sphere with a linear subspace.
Two generic $(d+1)$-dimensional linear subspaces of $\mathbb R^{2d+1}$ intersect in a line, hence two $d$-dimensional equatorial spheres in $\mathbb S^{2d}$ intersect at exactly two antipodal points.  
It follows that, for a generic choice of $a_1,a_2, \ldots,a_{d+1}, b_1,b_2,\ldots b_{d+1}$,
\[\mathbb P[V \cap U \neq \emptyset| \forall i: v_i=\pm a_i \textrm{ and } u_i =\pm b_i]=\frac{2}{2^{2d+2}}=\frac{1}{2^{2d+1}}.\]

Now, if each $a_i$ and $b_i$ is independently sampled from the uniform measure on $\mathbb S^{2d}$, then
\[\mathbb P[V \cap U \neq\emptyset]= \mathbb E_{a_i,b_i}[\mathbb P[V \cap U \neq \emptyset| \forall i: v_i=\pm a_i \textrm{ and } u_i= \pm b_i]].\]
Since $a_1,a_2, \ldots,a_{d+1}, b_1,b_2,\ldots b_{d+1}$ are generic almost surely, we can conclude that
 \[\mathbb E[\Cr(\phi)]=\frac{1}{2}{\binom{n}{d+1}}{\binom{n-(d+1)}{d+1}} \frac{1}{2^{2d+1}},\] as promised.
 \end{proof}

This concludes the proof of Theorem~\ref{Moon_Wendel}.

\begin{proof}[Proof of Corollary~\ref{double_count}]
Flores' theorem yields $\Cr_\iota([3]^{*(d+1)})=1$. There are ${\binom n3}^{d+1} $ distinct copies of it, and every pair that crosses lies in exactly $(n-2)^{d+1}$, which yields at least $ \frac{{\binom n3}^{d+1}}{(n-2)^{d+1}}$ different crossing pairs. Double counting as in the previous theorem yields the lower bound.

Consider an optimal spherical drawing of $\Delta^{n(d+1)-1}_{(d)}$ and sample a uniform random subcomplex isomorphic to $[n]^{*(d+1)}$.
There are 
\[{\binom{(d+1)n}{n,\,n,\,\dots,\,n}}\] such subcomplexes, where this denotes the multinomial coefficient. 
For each choice of $2d+2$ vertices spanning a pair of crossing simplices, there are $(d+1)!$ ways to place the first $d+1$ vertices as a simplex of $[n]^{*(d+1)}$ and another $(d+1)!$ ways for the second simplex. Finally, once the $2d+2$ vertices have been chosen and distributed in the different factors of $[n]$ in the join, there are \[{\binom{(d+1)(n-2)}{n-2,\,n-2,\,n-2,\,\dots,\,n-2}} \]
ways to complete this into a copy of $[n]^{*(d+1)}$ inside
 of $\Delta^{n(d+1)-1}_{(d)}$. Double counting, we get the inequality
 \[\Cr_\mathbb S(\Delta^{(d+1)n-1}_{(d)})\geq R_n \Cr_\mathbb S([n]^{*(d+1)}), \]
where 
\[
R_n \;=\; 
\frac{{\binom{(d+1)n}{n,\,n,\,\dots,\,n}}}
     {\big((d+1)!\big)^2 \, {\binom{(d+1)(n-2)}{n-2,\,n-2,\,\dots,\,n-2}} }.
\]

Taking the limit, we obtain
\[\lim_{n\to\infty}R_n=\frac{(d+1)^{2(d+1)}}{((d+1)!)^2}.\]

The complete complex has $\binom{(d+1)n}{d}$ $(d-1)$-faces
and $\binom{(d+1)n}{d+1}$ $d$-faces, whereas the partite
complex has $(d+1)n^d$ $(d-1)$-faces and $n^{d+1}$ $d$-faces.
Hence,
\[B^\mathbb S(\Delta_{(d)})\geq
\frac{1}{(d+1)!}B^\mathbb S([\infty]^{*(d+1)}).\]
In particular, Theorem~\ref{Moon_Wendel} yields
\[B^\mathbb S([\infty]^{*(d+1)})\leq
(d+1)!\,B^\mathbb S(\Delta_{(d)})
\leq\left(\frac{d+1}{4}\right)^{d+1}.\]
\end{proof}

The same inequality holds for $\Cr_\iota$ with the same proof. 

\section{Antipodal drawings}

We begin by proving theorems~\ref{antipodal} and~\ref{thm:dense_Z2}.

\begin{lemma} For any two $\mathbb Z_2$-equivariant maps $f_1,f_2\colon \mathbb S^d \to \mathbb S^{2d}$,

$$ |f_1( \mathbb S^d)\cap f_2( \mathbb S^d)| \geq 2. $$
\end{lemma}

\begin{proof}
The join $\mathbb S^d*\mathbb S^d$ is homeomorphic to $\mathbb S^{2d+1}$ and $\mathbb Z_2$ acts on it by mapping $tx+(1-t)y \mapsto -tx-(1-t)y$. This action is free.
By composing with the usual embedding of $\mathbb S^{2d}$ into $\mathbb R^{2d+1}$, we may think of $f_1$ and $f_2$ as taking values in $\mathbb R^{2d+1}$. We define the map 
$F \colon \mathbb S^d*\mathbb S^d\cong \mathbb S^{2d+1} \to \mathbb R^{2d+1}$ by 
\[F(tx+(1-t)y):=tf_1(x)-(1-t)f_2(y).\]
This is a $\mathbb Z_2$-equivariant map, so it must have a zero by the Borsuk-Ulam theorem. 
Let $x$, $y$ and $t$ be such that $tf_1(x)=(1-t)f_2(y)$. Taking norms on both sides, we obtain $t=t |f_1(x)|=(1-t)|f_2(y)|=(1-t)$, hence $t=\frac{1}{2}$ and $f_1(x)=f_2(y)$. This is to say that $f_1( \mathbb S^d) \cap f_2( \mathbb S^d)\neq \emptyset$ and, since the maps are $\mathbb Z_2$-equivariant, intersections come in antipodal pairs.
\end{proof}

\begin{proof}[Proof of Theorem~\ref{antipodal}]
Fix a generic $\mathbb Z_2$-equivariant drawing $f$ of $(\mathbb S^0)^{*n}_{(d)}$. Consider two disjoint $\mathbb Z_2$-invariant subspheres $\Sigma_1,\Sigma_2$ of $(\mathbb S^0)^{*(n)}_{(d)}$, both of them of dimension $d$, so that we have the following combinatorial equivalences $\Sigma_1\cong \Sigma_2\cong (\mathbb S^0)^{*(d+1)} \cong \mathbb S^d $. The map $f$ restricts to two $\mathbb Z_2$-equivariant maps $f_1\colon \Sigma_1 \to \mathbb S^{2d}$ and $f_2\colon \Sigma_2  \to \mathbb S^{2d}$. By the previous lemma, this implies that there are at least two crossings between these spheres. Moreover, crossings induced by different pairs of spheres of this kind are distinct: each crossing occurs between two $d$-faces, and each of these faces determines its $d+1$ antipodal vertex pairs.
The number of pairs of disjoint $(\mathbb S^0)^{*(d+1)}_{(d)}$ factors is $\frac{1}{2}{\binom{n}{d+1}}{\binom{n-(d+1)}{d+1}}$. All in all we have at least ${\binom{n}{d+1}}{\binom{n-(d+1)}{d+1}}$ distinct crossings.
The argument for the partite complex is exactly the same, except that there are $\frac{1}{2}n^{d+1} (n-1)^{d+1}$ vertex-disjoint pairs of $(\mathbb S^0)^{*(d+1)}_{(d)}$ factors.
\end{proof}

As was mentioned in the introduction, the construction yielding the corresponding upper bound is immediate if we consider our pairs of antipodal points to be in general position and every face is drawn geodesically (i.e., as the intersection between the sphere and the positive cone spanned by its vertices). In this case, following the previous proof, we have that $f_1(\Sigma_1)$ and $f_2(\Sigma_2)$ are equatorial $d$-spheres of $\mathbb S^{2d}$ which intersect in exactly two points. Organizing the computation as we did above, one obtains the upper bounds: $\Cr^{\mathbb Z_2}((\mathbb S^0)^{*n}_{(d)})\leq {\binom{n}{d+1}}{\binom{n-(d+1)}{d+1}}$, $\Cr^{\mathbb Z_2}([2n]^{*(d+1)})\leq n^{d+1} (n-1)^{d+1}$.

By extending the above argument, we obtain also a proof of Theorem~\ref{thm:dense_Z2}. However, we do not expect this result to be tight.

\begin{proof}[Proof of Theorem~\ref{thm:dense_Z2}]
Again, every pair of vertex-disjoint $\mathbb Z_2$-invariant subspheres isomorphic to $(\mathbb S^0)^{*(d+1)}$ induces at least two crossings in any generic equivariant drawing. As before, crossings induced by different pairs of this kind are distinct.

Since the action is free, no face contains an antipodal vertex pair, so we can view $X_n$ as a subcomplex of $(\mathbb S^0)^{*n}_{(d)}$. For each of its $\binom{n}{d+1}$ coordinate subspheres isomorphic to $(\mathbb S^0)^{*(d+1)}$, count how many facets belong to $X_n$. This number is even. Let $A_1$ denote the number of subspheres such that all $2^{d+1}$ facets are contained in $X_n$, and $A_2$ the number such that at most $2^{d+1}-2$ facets are contained in $X_n$. Then $A_1+A_2=\binom{n}{d+1}$ and
\[2^{d+1}A_1+(2^{d+1}-2)A_2\geq|X_n(d)|=(\alpha+o(1))2^{d+1}\binom{n}{d+1}.\]
Consequently, $A_1\geq\left(\frac{\alpha2^{d+1}-(2^{d+1}-2)}{2}+o(1)\right)\binom{n}{d+1}$. Only $O_d(n^{2d+1})$ pairs of coordinate subspheres share a vertex, so the number of vertex-disjoint pairs fully contained in $X_n$ is at least $\binom{A_1}{2}-O_d(n^{2d+1})$. Thus
\[\Cr^{\mathbb Z_2}(X_n)\geq\left(\frac{(\alpha2^{d+1}-(2^{d+1}-2))^2}{4}-o(1)\right)\binom{n}{d+1}^{\!2}.\]

For the normalization, observe that each $(d-1)$-face is contained in at most $2(n-d)$ $d$-faces. Double counting incidences gives $|X_n(d-1)|\geq\frac{d+1}{2(n-d)}|X_n(d)|$, and hence
\[\Cr^{\mathbb Z_2}(X_n)\frac{|X_n(d-1)|^{d+1}}{|X_n(d)|^{d+2}}\geq\frac{\Cr^{\mathbb Z_2}(X_n)}{|X_n(d)|}\left(\frac{d+1}{2(n-d)}\right)^{d+1}.\]
Letting $n\to\infty$ proves the claim. \end{proof}

We also have the following corollary of Theorem~\ref{antipodal}.

\begin{proof}[Proof of Corollary~\ref{hill_conjecture}]
 Every image of every $(\mathbb S^0)^{*2}$-factor is a (not necessarily simple) curve that defines two open sets $A$ and $-A$ in the sphere, and for each vertex $v \in A$, $-v\in -A$. Thus, any arc joining $v$ to $-v$ will intersect $\partial A$ at least once; this is an intersection with the $(\mathbb S^0)^{(*2)}$-factor.
In the case of $2n$ vertices, in total we have ${\binom n2}{\binom{n-2}2}$ intersections coming from pairs of $(\mathbb S^0)^{(*2)}$ factors and $n {\binom{n-1}2}$ coming from edges joining antipodal points intersecting $(\mathbb S^0)^{*2}$ factors. In the case of $2n+1$ vertices, one of the vertices must stay unmatched; the computation is almost identical except that we now add another $n {\binom{n-1}2}$ crossings coming from intersections between $(\mathbb S^0)^{*2}$ factors and edges incident to the unmatched vertex. Moreover, the unmatched vertex and any two antipodal pairs form a $K_5$, which has at
least one crossing between independent edges that was not counted before. In total we have
\[\binom n2\binom{n-2}{2}+2n\binom{n-1}{2}+\binom n2=\frac{n^2(n-1)^2}{4}
\]
crossings, as desired.
\end{proof}

Again, drawing the edges as geodesic arcs makes the inequalities in the previous argument into equalities. 
However, we still need to deal with intersections between pairs of arcs joining antipodal points as we explained in the introduction.
To obtain the Harary-Hill bound, we need no new intersection to appear, and here the choice of antipodal arcs is important. 
It was first observed by Mohar that this is possible. Specifically, Lemma 7 in~\cite{bojan2020_hill_conjecture} claims that given a semicircle, one can construct an arbitrary number of pairwise-disjoint semicircles that are geometrically close to the starting one. We provide a detailed proof below.

In what follows, a semicircle is always contained in an equatorial circle (the intersection of a linear two-dimensional space with the unit sphere).
Consider a set of points $q_1,q_2,\ldots q_n$ near the point $(0,0,1)$ on a strictly convex curve on the hyperplane $z=1$ (e.g. on the curve $y=x^2,z=1$). 
For each $i$, the plane through the origin $q_i^\perp$, which is perpendicular to $q_i$, intersects the sphere in an equator $c_i=q_i^\perp \cap \mathbb S^2$. Assume that the points are ordered according to the $x$-coordinate, and orient the circles by projecting them to the $xy$-plane and giving this plane the counterclockwise orientation. 
On the front of the sphere (the set $y\geq 0$), we see that the circle $c_1$ intersects the other circles in the order $c_2,c_3,\ldots$. We define the semicircle $s_1\subset c_1$ to begin before $p_{12}=c_1\cap c_2$, and the semicircle $s_2$ just after $p_{12}$. 
Notice that this implies $s_1\cap s_2=\emptyset$. Similarly, once we have defined $s_{i-1}$, we define $s_{i}$ to begin between the intersection $c_{i-1}\cap c_{i}$ and the intersection $c_i\cap c_{i+1}$. This construction yields no intersection between semicircles and it can be made arbitrarily close to a starting semicircle. In this construction, it is fundamental that the starting set of circles is dual to a point set in convex position.

\section{Facewise spherical drawings}

Here, we study facewise spherical drawings of $d$-dimensional complexes on $\mathbb S^{2d}$.

\subsection{Drawing the complete complex}

The purpose of this section is to prove a lower bound on $\Cr_{\mathbb S^{2d}}(\Delta_{(d)}^{n-1})$ conditionally on Conjecture~\ref{gubt}. Our notation and definitions in this part are essentially borrowed from~\cite{streltsova2025spherical}. Let $V\subset \mathbb R^m$ be a set of $n$ vectors in linear general position. We say that a non-empty subset $U \subset V$ supports a \textit{dependency} if there exist nonzero numbers $\{\lambda_u: u \in U\}$ such that $\sum_{u\in U} \lambda_u u=0$; evidently, multiplying the coefficients by any nonzero real number gives another dependency on $U$. Consider two such dependencies $\lambda,\mu\in\mathbb R^U$. We say they are equivalent whenever their sets of negative coefficients coincide: $\{u:\lambda_u<0\}=\{u:\mu_u<0\}$. Notice that equivalence classes of dependencies come in pairs: if $\lambda$ is a dependency on $U$, then so is $-\lambda$, and they are not equivalent. 

Let $f^*_{s,t}(V)$ be the number of dependency sign patterns on $V$ with exactly $s$ nonzero coefficients and $t$ negative coefficients.

To see the relevance to crossing numbers of these concepts, let us set $m=2d+1$ and $s=m+1$.
Consider a set $U\subset V$ with $|U|=s=m+1=2d+2$. Since every $2d+1$ vectors of $U$ are linearly independent, its dependency space is one-dimensional, and $U$ gives rise to two classes of dependencies, with $t$ and $s-t$ negative scalars, respectively.

It follows that 
\[\sum_{t=0}^{2d+2} f^*_{2d+2,t}(V)=2 {\binom{n}{2d+2}}.\]
By the symmetry $f^*_{s,t}(V)=f^*_{s,s-t}(V)$, this can be rewritten as 
\[f^*_{2d+2,\leq d}(V)+\frac{1}{2}f^*_{2d+2,d+1}(V)= {\binom{n}{2d+2}},\] where \[f^*_{s,\leq t}(V):=\sum_{j=0}^t f^*_{s,j}(V).\]

Notice that in the case $t=d+1$, the sets $U_-:=\{u: \lambda_u<0\}$ and $U_+:=\{u: \lambda_u>0\}$ have $d+1$ vertices each. In this case the cones $\cone(U_-)$ and $\cone(U_+)$ intersect in a ray from the origin; equivalently, $\sph(U_+)\cap\sph(U_-)\neq\emptyset$. Hence, this pair of dependencies corresponds to one crossing. We record this as a lemma.

\begin{lemma}\label{crossing-dependency} Given $V \subset \mathbb S^{2d}\subset \mathbb R^{2d+1}$, a configuration of $n$ vectors in linear general position, let $g\colon \Delta_{(d)}^{n-1}\to \mathbb S^{2d}$ be the induced generic facewise spherical drawing. Then \[\Cr(g)=\frac{1}{2}f^*_{2d+2,d+1}(V)={\binom{n}{2d+2}}-f^*_{2d+2,\leq d}(V).\]
\end{lemma}

Given a labelled configuration of $n$ vectors $V \subset \mathbb R^m$, we abuse notation and denote also by $V$ the matrix that has these vectors as columns. Two labelled vector configurations $V \subset \mathbb R^m$, $W\subset\mathbb R^{n-m}$, with $|V|=|W|=n$, are \emph{Gale dual} if $\operatorname{rank}V=m$, $\operatorname{rank}W=n-m$, and $VW^T=0$. Equivalently, the rows of $W$ form a basis of $\ker V$. In this case we denote $W$ by $V^*$. Notice that this is only defined up to an invertible linear transformation (if $W$ is Gale dual of $V$, it is also Gale dual of $AV$ for a linear isomorphism $A$, and similarly $BW$ is Gale dual of $V$ for a linear isomorphism $B$). We denote by $V^*$ any Gale dual of $V$ (the properties that concern us do not depend on this choice).

The following observation is from~\cite{welzl2001entering} and its relevance to crossing numbers was observed in~\cite{wagner2003ksets}.

\begin{lemma}\label{lem:gale-welzl} For any rank-$m$ vector configuration $V$ in linear general position on $n>m$ labelled vectors and any Gale dual $V^*$,
\[f^*_{s,t}(V)=f_{n-s,t}(V^*).\]
\end{lemma}
 
We use the arrangement formulation of the Upper Bound Theorem and its extension to sublevels, Conjecture~\ref{gubt}, stated in the introduction. Recall that $V_{n,m}$ is a rank-$m$ cyclic vector configuration, corresponding after positive rescaling to affine dimension $m-1$.

The case $m=n-3$ of this arrangement conjecture (equivalently, the statement about dependency patterns of rank-three vector configurations) was verified in~\cite{streltsova2025spherical}. Together with the previous two lemmas it implies the Hill conjecture for spherical drawings of the complete graph. For the number of vertices at a sublevel, Wagner proved the conjectured bound up to a factor of $4$ in arbitrary rank.

Let $\gamma(t):=(1,t,t^2,\ldots,t^{2d})\in\mathbb R^{2d+1}$. For reals $t_1<t_2<\cdots<t_m$, consider the \emph{alternating moment configuration}
\[W_m:=\bigl\{w_i:=(-1)^i\gamma(t_i)/\|\gamma(t_i)\| : 1\leq i\leq m\bigr\}\subset\mathbb S^{2d}.\]
By the Vandermonde determinant, it is in linear general position and, for a generic choice of the $t_i$, the induced facewise spherical drawing $g_m\colon\Delta^{m-1}_{(d)}\to\mathbb S^{2d}$ is generic.

\begin{proposition}\label{prop:alternating}
For all $d\geq1$ and $m\geq 2d+2$,
\[\Cr(g_m)=\binom{\lfloor m/2\rfloor}{d+1}\binom{\lfloor(m-1)/2\rfloor}{d+1}.\]
Moreover, $\Cr(g_m)=\binom{m}{2d+2}-f_{m-2d-2,\leq d}(V_{m,m-2d-1})$. Consequently,
\[\Cr_{\mathbb S^{2d}}(\Delta^{m-1}_{(d)})\leq\binom{\lfloor m/2\rfloor}{d+1}\binom{\lfloor(m-1)/2\rfloor}{d+1},\]
with equality if Conjecture~\ref{gubt} holds.
\end{proposition}

\begin{proof}
Let $i_0<\cdots<i_{2d+1}$ be indices in $[m]$. The divided difference of order $2d+1$ annihilates every polynomial of degree at most $2d$, so the vectors $\gamma(t_{i_0}),\ldots,\gamma(t_{i_{2d+1}})$ satisfy the dependency with coefficients $\prod_{k\neq j}(t_{i_j}-t_{i_k})^{-1}$, whose sign is $(-1)^{j+1}$. Thus the dependency among $w_{i_0},\ldots,w_{i_{2d+1}}$ has signs $(-1)^{i_j+j+1}$. By Lemma~\ref{crossing-dependency}, these $2d+2$ vertices span a crossing pair exactly when $i_j-j$ is even for exactly $d+1$ indices $j$.

Put $k_j:=i_j-j$. This is a bijection between $(2d+2)$-subsets of $[m]$ and nondecreasing sequences of length $2d+2$ in $\{1,\ldots,m-2d-1\}$. A nondecreasing sequence with exactly $d+1$ even terms is determined by its multiset of $d+1$ even values and its multiset of $d+1$ odd values. There are $\lfloor(m-2d-1)/2\rfloor$ even and $\lceil(m-2d-1)/2\rceil$ odd values available, so
\[\Cr(g_m)=\binom{\lfloor(m-2d-1)/2\rfloor+d}{d+1}\binom{\lceil(m-2d-1)/2\rceil+d}{d+1}=\binom{\lfloor(m-1)/2\rfloor}{d+1}\binom{\lfloor m/2\rfloor}{d+1}.\]

For the second identity, put $c_i:=\prod_{j\neq i}(t_i-t_j)^{-1}$, so that $\operatorname{sign}c_i=(-1)^{m-i}$. The identity $\sum_i c_it_i^{a}=0$ for $0\leq a\leq m-2$ shows that $\{c_i\gamma(t_i)\}$ is a Gale dual of $\{(1,t_i,\ldots,t_i^{m-2d-2})\}$, and the latter has the same sign patterns as $V_{m,m-2d-1}$. Up to positive rescaling and a global sign, which affect neither $f^*$ nor $f$, the configuration $\{c_i\gamma(t_i)\}$ is $W_m$. Lemmas~\ref{crossing-dependency} and~\ref{lem:gale-welzl} imply
\[\Cr(g_m)=\binom{m}{2d+2}-f^*_{2d+2,\leq d}(W_m)=\binom{m}{2d+2}-f_{m-2d-2,\leq d}(V_{m,m-2d-1}).\]

Finally, let $V\subset\mathbb S^{2d}$ induce any generic facewise spherical drawing of $\Delta^{m-1}_{(d)}$. By the same two lemmas, its crossing number equals $\binom{m}{2d+2}-f_{m-2d-2,\leq d}(V^*)$, where $V^*$ has rank $m-2d-1$ and is in linear general position. Conjecture~\ref{gubt}, applied with $t=d=(m-(m-2d-1)-1)/2$, gives $f_{m-2d-2,\leq d}(V^*)\leq f_{m-2d-2,\leq d}(V_{m,m-2d-1})$. Hence the crossing number is at least $\Cr(g_m)$.
\end{proof}

\begin{proof}[Proof of Proposition~\ref{prop:lew-spherical}]
Assuming Conjecture~\ref{gubt}, Proposition~\ref{prop:alternating} gives \[\Cr_{\mathbb S^{2d}}(\Delta_{(d)}^{m-1})\sim m^{2d+2}/(2^{2d+2}((d+1)!)^2)\] as $m\to\infty$. Since $\Delta_{(d)}^{m-1}$ has $\binom md$ faces of dimension $d-1$ and $\binom m{d+1}$ faces of dimension $d$, we obtain
\[B^{\mathbb S}(\Delta_{(d)})=\lim_{m\to\infty}\Cr_{\mathbb S^{2d}}(\Delta_{(d)}^{m-1})\frac{\binom md^{d+1}}{\binom m{d+1}^{d+2}}=\frac{(d+1)^d}{2^{2d+2}d!}.\]
\end{proof}

\subsection{Geometric crossing numbers of the random simplicial complex}

Our goal in this section is to prove Theorem~\ref{thm:rectlinear_random_complex}.
Recall that the Linial-Meshulam random complex $\Delta(n,p)$ is the $d$-dimensional simplicial complex on $n$
labeled vertices whose $(d-1)$-skeleton is complete and in which each
$d$-face is included independently with probability $p$.
Let us start by proving that the random complex $\Delta(n,p)$ is asymptotically ``uniform-like''. This fact is standard, but we have not found it explicitly stated anywhere in the form that we require.

\begin{lemma}\label{lem:uniformity}
Let $k> d\geq 1$ be fixed, and consider the random $d$-dimensional simplicial complex $X\sim \Delta_d(n,p)$ with vertex set $[n]$. Suppose that
$p n^d \rightarrow \infty .$
Then, with high probability as $n\to\infty$, the following holds uniformly
over all pairwise disjoint sets \[U_0,U_1,\dots,U_d\subseteq [n]\quad\text{with}\quad |U_i|\geq \frac{n}{k}.\] Namely, \[ F_X(U_0,\dots,U_d)=p\prod_{i=0}^d |U_i|\,(1+o(1)),\] where \[F_X(U_0,\dots,U_d):=\left|\Big\{\{u_0,\dots,u_d\}\in X(d):u_i\in U_i \text{ for every } i=0,\dots,d\Big\}\right|.\]
The $o(1)$ term is uniform over all such choices of $U_0,\dots,U_d$.
\end{lemma}

\begin{proof}
Fix pairwise disjoint sets \[U_0,\dots,U_d\subseteq [n]\quad\text{with}\quad |U_i|\geq \frac{n}{k}.\] Every choice of vertices $u_i\in U_i$, $i=0,\dots,d$, determines a potential $d$-face $\{u_0,\dots,u_d\}.$ Since the sets $U_0,\dots,U_d$ are pairwise disjoint, all these faces are distinct. Therefore \[F_X(U_0,\dots,U_d)\sim\operatorname{Bin}\left(\prod_{i=0}^d |U_i|,p\right).\] In particular,\[\mu(U_0,\dots,U_d):=\mathbb E F_X(U_0,\dots,U_d)=p\prod_{i=0}^d |U_i|.\]
By the lower bound on the sizes of the \(U_i\)'s,\[\mu(U_0,\dots,U_d)\geq p\left(\frac{n}{k}\right)^{d+1}.\]

We now prove concentration uniformly over all admissible choices of
$U_0,\dots,U_d$. Let $\varepsilon_n>0$ be defined by \[\varepsilon_n := (p n^d)^{-1/4}.\]
Since $p n^d\to\infty$, we have $\varepsilon_n=o(1)$. By the Chernoff bound, for fixed $U_0,\dots,U_d$, \[\mathbb P\left(\left|F_X(U_0,\dots,U_d)-\mu(U_0,\dots,U_d)\right|>\varepsilon_n \mu(U_0,\dots,U_d)\right)\leq2\exp\left(-c\varepsilon_n^2\mu(U_0,\dots,U_d)\right)\] for some absolute constant $c>0$. Using the lower bound on $\mu$, this is at most \[2\exp\left(-c\varepsilon_n^2p\left(\frac{n}{k}\right)^{d+1}\right).\]
We have that \[\varepsilon_n^2 p n^{d+1}=(p n^d)^{-1/2} p n^{d+1}=n(p n^d)^{1/2}.\] Since $p n^d\to\infty$, we have \[\varepsilon_n^2 p n^{d+1}\gg n.\]

There are at most $2^{(d+1)n}$ possible choices of
$(U_0,\dots,U_d)$. Hence, by the union bound, the probability that the
desired estimate fails for at least one admissible $(d+1)$-tuple is at most \[2^{(d+1)n}\cdot2\exp\left(-c\varepsilon_n^2 p\left(\frac{n}{k}\right)^{d+1}\right).\]
Because \[\varepsilon_n^2 p n^{d+1}\gg n,\] this quantity tends to $0$. Therefore, with high probability, for every admissible pairwise disjoint collection $U_0,\dots,U_d$, \[\left|F_X(U_0,\dots,U_d)-p\prod_{i=0}^d |U_i|\right|\leq\varepsilon_n p\prod_{i=0}^d |U_i|.\]
Since $\varepsilon_n=o(1)$, this gives\[F_X(U_0,\dots,U_d)=p\prod_{i=0}^d |U_i|\,(1+o(1)),\] uniformly over all admissible choices of $U_0,\dots,U_d$.
\end{proof}

Given a $d$-dimensional simplicial complex on $n$ vertices and a point set
$V \subset \mathbb{R}^{2d}$ with $|V|=n$, a map $f \colon X_{(0)}\to V$ extends
uniquely to a facewise affine map of the whole complex, which we denote also by $f$.

Two $d$-faces $\sigma,\tau$ in a facewise affine drawing cross if and only if the $2d+2$ vertices of $\sigma\cup\tau$ are distinct and their Radon partition is $(\sigma,\tau)$. Since the Radon partition of a $(2d+2)$-tuple is determined by the orientations of all $(2d+1)$-element subsets, the crossing number
$\Cr_{\mathbb{R}^{2d}}(f)$ depends on $f \colon [n]\to V$ only through its \emph{order type} (chirotope) $\chi$ (the relevant definitions can be found, for example, in~\cite{goodman_order_types}).

Given $k$ disjoint subsets $V_1,\dots,V_k \subset V$, a transversal for the tuple $(V_1,\dots,V_k)$ is any tuple $(v_1,\dots,v_k)$ of elements of $V$ such that
$v_i \in V_i$ for all i. We say that the $k$-tuple of parts $(V_1, \dots, V_k)$ has \textit{same-type transversals} if all of its transversals have the same order type. The main ingredient in our proof of Theorem~\ref{thm:rectlinear_random_complex} is the following direct consequence of the semi-algebraic regularity lemma of Fox, Pach and Suk~\cite{fox_semialgebraic}.

\begin{theorem}\label{thm:semi_algebraic_regularity}
For every $0<\eta<1/2$, there exists an integer $M=M(d,\eta)$ such that every sufficiently large finite point set $V\subset\mathbb R^{2d}$ in general position admits an equitable partition $V=V_1\cup\cdots\cup V_K$, where $\max\{2d+2,\eta^{-1}\}\leq K\leq M$, such that all but at most $\eta\binom K{2d+2}$ choices of $2d+2$ distinct parts have same-type transversals. Here, an equitable partition is one whose part sizes differ by at most one.
\end{theorem}

\begin{proof}[Proof of Theorem~\ref{thm:rectlinear_random_complex}]
We first prove the lower bound. Fix a rectilinear drawing of $X$ given by a bijection $f:[n]\to V$, where
\[V=\{v_1,\dots,v_n\}\subset \mathbb R^{2d}\]
is the corresponding vertex set in affine general position. Suppose that $f$ and $V$ are such that $\mathrm{cr}_{\mathbb R^{2d}}(f)=\mathrm{cr}_{\mathbb R^{2d}}(X)$. Apply Theorem~\ref{thm:semi_algebraic_regularity} to $V$, with parameter $\eta>0$, obtaining an equitable partition
\[V=V_1\cup\cdots\cup V_K.
\]
Here $K$ is bounded above by $M=M(d,\eta)$. We say that a collection of $2d+2$ distinct parts of this partition is \textit{same-type} if the parts have same-type transversals, and we say that it is \textit{non-same-type} otherwise.

Call a crossing pair of $d$-faces in the complete complex \textit{bad} if either two of its $2d+2$ vertices lie in the same part $V_i$, or the $2d+2$ parts it uses form a non-same-type tuple.

The number of crossing pairs using some part twice is at most
\[O_d\left(\frac{n^{2d+2}}{K}\right).\]
The number of crossing pairs whose corresponding $2d+2$ parts form a non-same-type tuple is at most
\[O_d(\eta n^{2d+2}).\]
Thus, if $\mathcal T$ denotes the collection of unordered pairs
\[(I,J),\qquad I,J\subseteq [K], \quad |I|=|J|=d+1,\quad I\cap J=\varnothing,\]
such that $I\cup J$ is same-type and every \(I\)-transversal \(d\)-simplex crosses every \(J\)-transversal \(d\)-simplex, then
\[\sum_{(I,J)\in\mathcal T}\prod_{i\in I\cup J} |V_i|\geq\operatorname{cr}_{\mathbb R^{2d}}(\Delta_{(d)}^{n-1})-O_d\left(\frac{n^{2d+2}}{K}\right)-O_d(\eta n^{2d+2}).\]

Now we apply Lemma~\ref{lem:uniformity}. Since $K$ is fixed and the partition is equitable, every part has size
\[|V_i|=(1+o(1))\frac nK.\]
Because $p n^d\to\infty$, with high probability, uniformly over all disjoint $(d+1)$-tuples of parts,
\[F_X(V_{i_0},\dots,V_{i_d})=p\prod_{r=0}^d |V_{i_r}|(1+o(1)).
\]

Fix $(I,J)\in\mathcal T$. Every $d$-face of $X$ using one vertex from each part $V_i$ with $i\in I$ crosses every $d$-face of $X$ using one vertex from each part $V_j$ with $j\in J$. Therefore, this pair $(I,J)$ contributes \[F_X((V_i)_{i\in I})\, F_X((V_j)_{j\in J})=p^2\prod_{i\in I\cup J}|V_i|\,(1+o(1))\]
crossings of $X$. Summing over $(I,J)\in\mathcal T$, we obtain
\[\operatorname{cr}_{\mathbb R^{2d}}(X)=\operatorname{cr}_{\mathbb R^{2d}}(f)\geq p^2\sum_{(I,J)\in\mathcal T}\prod_{i\in I\cup J}|V_i|\,(1+o(1)).\]
Hence \[\operatorname{cr}_{\mathbb R^{2d}}(X)\geq p^2\operatorname{cr}_{\mathbb R^{2d}}(\Delta_{(d)}^{n-1})-O_d\left(\frac{p^2 n^{2d+2}}{K}\right)-O_d(\eta p^2 n^{2d+2})-o(p^2 n^{2d+2}).\] 
Letting $\eta\to0$, and then $K=K(d,\eta)\to\infty$ accordingly, we get
\[\operatorname{cr}_{\mathbb R^{2d}}(X)\geq p^2\operatorname{cr}_{\mathbb R^{2d}}(\Delta_{(d)}^{n-1})-o_{\mathbb P}(p^2 n^{2d+2}).\]

It remains to prove the upper bound. Take an optimal rectilinear drawing of the complete complex $\Delta_{(d)}^{n-1}$, so that the number of crossing pairs of $d$-faces is exactly $\operatorname{cr}_{\mathbb R^{2d}}(\Delta_{(d)}^{n-1}).$
Draw $X$ on the same vertex set by mapping the vertices of $X$ in a fixed manner, independently of the randomness involved in the construction of the complex. Let $\mathcal C$ be the set of crossing pairs of $d$-faces in this drawing of $\Delta_{(d)}^{n-1}$. We have \[|\mathcal C|=\operatorname{cr}_{\mathbb R^{2d}}(\Delta_{(d)}^{n-1}).\] For $c=\{\sigma,\tau\}\in\mathcal C$, let $I_c$ be the indicator that
both $\sigma$ and $\tau$ belong to $X$. The number of crossings in this fixed drawing of $X$ is \[Z:=\sum_{c\in\mathcal C} I_c.\] Hence\[\mathbb E Z=p^2 \operatorname{cr}_{\mathbb R^{2d}}(\Delta_{(d)}^{n-1}).\]

We claim that \[Z=p^2 \operatorname{cr}_{\mathbb R^{2d}}(\Delta_{(d)}^{n-1})+o_{\mathbb P}(p^2 n^{2d+2}),\] where $o_{\mathbb P}\bigl(p^2 n^{2d+2}\bigr)$ denotes an error term which goes to $0$ in probability after being divided by $p^2n^{2d+2}$.
Indeed, \[\operatorname{Var}(Z)\leq O_d(p^2 n^{2d+2})+O_d(p^3 n^{3d+3}).\] The first term comes from the variances of the individual indicators. For the covariance term, two indicators can be dependent only if their two crossing pairs share a $d$-face. There are $O_d(n^{d+1})$ choices for the shared $d$-face, and then $O_d(n^{d+1})$ choices for each of the two other $d$-faces, giving $O_d(n^{3d+3})$ dependent pairs, each with covariance at most $O(p^3)$.

Since $p n^d\to\infty$, we have that
\[\frac{p^2 n^{2d+2}}{p^4 n^{4d+4}}=o(1)\quad\text{and}\quad \frac{p^3 n^{3d+3}}{p^4 n^{4d+4}}=o(1).\]
Hence, \[\operatorname{Var}(Z)=o(p^4 n^{4d+4}).\]
By Chebyshev's inequality, \[Z=p^2 \operatorname{cr}_{\mathbb R{2d}}(\Delta_{(d)}^{n-1})+o_{\mathbb P}(p^2 n^{2d+2}).\]
Since $\operatorname{cr}_{\mathbb R^{2d}}(X)\leq Z$ for every realization, we get \[\operatorname{cr}_{\mathbb R^{2d}}(X)\leq p^2 \operatorname{cr}_{\mathbb R^{2d}}(\Delta_{(d)}^{n-1})+o_{\mathbb P}(p^2 n^{2d+2}).\]

Combining the lower and upper bounds gives
\[\operatorname{cr}_{\mathbb R^{2d}}(X)=p^2\operatorname{cr}_{\mathbb R^{2d}}(\Delta_{(d)}^{n-1})+o_{\mathbb P}(p^2 n^{2d+2}).\]

\end{proof}

The same argument applies to facewise spherical drawings. For vertices in $\mathbb S^{2d}\subset\mathbb R^{2d+1}$, replace affine order types by the signs of the determinants of $(2d+1)$ vertex vectors. These signs determine the signed dependence of every $(2d+2)$-tuple and hence its crossing pair, as in Lemma~\ref{crossing-dependency}. Since these determinant conditions are semi-algebraic of complexity bounded in terms of $d$, the same regularity argument applies. Each pair of independent $d$-faces crosses at most once, so the counting and variance estimates remain unchanged. We therefore obtain the following spherical version of Theorem~\ref{thm:rectlinear_random_complex}.

\begin{theorem}\label{thm:spherical_random_complex}
Fix \(d\geq 1\) and let \(X\sim \Delta_d(n,p)\).
If\[p n^d\rightarrow \infty,\] then \[\operatorname{cr}_{\mathbb S^{2d}}(X)=p^2\operatorname{cr}_{\mathbb S^{2d}}(\Delta_{(d)}^{n-1})+o_{\mathbb P}\bigl(p^2 n^{2d+2}\bigr),\] where $o_{\mathbb P}\bigl(p^2 n^{2d+2}\bigr)$ denotes an error term which, when divided by $p^2n^{2d+2}$, goes to $0$ in probability.

\end{theorem}

\section{The midrange}\label{limits}

\begin{proof}[Proof of Lemma~\ref{lem:dense_seeds}]

Fix a sufficiently large $a$, and choose $n$ such that
\[
a_n\leq a<a_{n+1}.
\]
For every positive integer $k$, take $k$ disjoint copies of $X_{n+1}$, drawn so that different copies do not intersect. The resulting complex has $ks_{n+1}$ $(d-1)$-faces, $kt_{n+1}$ $d$-faces, and at most $k\Cr_\iota(X_{n+1})$ crossing pairs.

Since $a\geq a_n$,
\[
b_d(a)
\leq
\Cr_\iota(X_{n+1})\frac{s_{n+1}^{d+1}}{t_{n+1}^{d+2}}
\left(\frac{a_{n+1}}{a_n}\right)^{d+2}.
\]
Taking the $\limsup$ as $a\to\infty$ and using
$a_{n+1}/a_n\to1$ proves the claim.
\end{proof}

\begin{proof}[Proof of Lemma~\ref{midrange_sampling}]
Fix $\varepsilon\in(0,1)$. We may assume that $C\geq1$, since increasing $C$ weakens the hypotheses. Firstly, for every positive $a$ and every $d$-complex $Z$, such that $|Z(d)|\geq a|Z(d-1)|$, we have
\[
\Cr_\iota(Z)\geq b_d(a)a^{d+2}|Z(d-1)|.
\]
To see this, consider $k$ disjoint copies of $Z$. By definition, $\kappa_d(k|Z(d-1)|,ak|Z(d-1)|)$ are bounded from above by $k\Cr_\iota(Z)$. Dividing this inequality by $a^{d+2}k|Z(d-1)|$ and taking the limit when $k\to \infty$ yields the claimed inequality. Let $\delta=\varepsilon/(100(d+2))$ and choose $a_0$ large enough that $b_d(a_0)\geq(1-\varepsilon/10)B_d$. Set $A=(1+4\delta)a_0$ and choose $N>A$. Notice that the lower bound on $A$ depends only on $d,\varepsilon$ and $C$.

Assume that $X$ satisfies the hypothesis of the lemma, put $s=|X(d-1)|$, $t=|X(d)|$, and as in the proof of the crossing lemma, let $Y$ be a random induced subcomplex in which every vertex is included independently at random with probability $p=\frac{As}{t}$.
As in the crossing lemma, we have 
\[\mathbb E[|Y(d-1)|]=p^ds,\]
\[\mathbb E[|Y(d)|]=p^{d+1}t=A p^ds,\] 
\[\mathbb E[\Cr_\iota(Y)]\leq p^{2d+2}\Cr_\iota(X).\]
For a face $\tau$, let $\deg_r(\tau)$ be the number
of $r$-faces containing it. Two distinct
$(d-1)$-faces whose intersection has $j$ vertices
have covariance at most $p^{2d-j}$, whereas disjoint
faces have covariance zero. Moreover,
$\deg_{d-1}(\tau)\leq |X(d-j-1)|$ when $|\tau|=j$,
and
$\sum_{\tau\in X(j-1)}\deg_{d-1}(\tau)=\binom dj s$.
Hence, 
\[
\var(|Y(d-1)|)
\leq p^ds+s\sum_{j=1}^{d-1}\binom dj |X(d-j-1)|p^{2d-j}.
\]
By assumption, $|X(d-j-1)|\leq Cs(s/t)^j$. Therefore
\[
\frac{\var(|Y(d-1)|)}{(p^ds)^2}\leq \frac1{p^ds}+C\sum_{j=1}^{d-1}\binom dj A^{-j}\leq \frac1{p^ds}+\frac{2^dC}{A}.
\]
Similarly,
\[
\var(|Y(d)|)\leq p^{d+1}t+t\sum_{j=1}^{d}\binom{d+1}{j}|X(d-j)|p^{2d+2-j},
\]
so 
\[
\frac{\var(|Y(d)|)}{(p^{d+1}t)^2}\leq \frac1{p^{d+1}t}+\frac{2^{d+1}C}{A}.
\]

By the Kruskal--Katona theorem, there exists a constant $c_d>0$ such that
$t\leq c_ds^{(d+1)/d}$. Therefore
$p^ds=A^ds^{d+1}/t^d\geq c_d^{-d}A^d$ and
$p^{d+1}t\geq c_d^{-d}A^{d+1}$.
For $A$ large enough, both normalized variances are at most
$\varepsilon\delta^2/20$.
Applying Chebyshev's inequality to both normalized variance bounds and taking a union bound we conclude that the event 
\[E:=\left\{\bigl||Y(d-1)|-p^ds\bigr|\leq\delta p^ds \ \textrm{ and }\  \bigl||Y(d)|-p^{d+1}t\bigr|\leq\delta p^{d+1}t\right\}\] 
occurs with probability at least $1-\varepsilon/10$.
On $E$, $|Y(d-1)|\geq(1-\delta)p^ds$ and
\[
\frac{|Y(d)|}{|Y(d-1)|}\geq A\frac{1-\delta}{1+\delta}\geq a_0.
\]
Hence 
\[\mathbb E[\Cr_\iota(Y)]\geq \mathbb P(E)b_d(a_0)a_0^{d+2}(1-\delta)p^ds\geq\frac{(1-\varepsilon/10)^2(1-\delta)} {(1+4\delta)^{d+2}} B_d A^{d+2}p^ds.
\]
Since $\delta=\varepsilon/(100(d+2))$, $\mathbb E[\Cr_\iota(Y)]\leq
p^{2d+2}\Cr_\iota(X)$ and $p=As/t$, we obtain
\[
\Cr_\iota(X)\frac{s^{d+1}}{t^{d+2}}>(1-\varepsilon)B_d,
\]
as required.
\end{proof}

Before presenting the proof of Lemma~\ref{midrange_blowup}, we recall that throughout $q$ is some number satisfying $d< q<d+1$. We are going to work with $q$-equitable complexes, and recall that the functions $\kappa_d'$ and $b_d'$ are defined using this $q$ which is fixed through the proof.

We need a generalization of the blow-up operation from graph theory to simplicial complexes. Specifically, if $X$ is a simplicial complex, the $m$-\emph{blow-up} of $X$ is the simplicial complex, denoted by $X\centerdot m$, in which each vertex is replaced by $m$ vertices, and each simplex of dimension $i$ is replaced by a copy of the complete $(i+1)$-partite complex, with $m$ vertices on each part. In other words, each $i$-simplex is replaced with the iterated join $[m]^{*(i+1)}$. Notice that the blow-up comes with a simplicial map $\varphi\colon X\centerdot m\to X$ in which the fiber of each vertex in $X$ is the set of $m$ vertices it got replaced by in $X\centerdot m$.

\begin{proof}[Proof of Lemma~\ref{midrange_blowup}]
Let $\varepsilon>0$ and choose $\eta\in(0,1)$ with $(1+\eta)(1+2\eta)^{d+2}<1+\varepsilon$. Recall that $B_d'\geq B_d>0$. Moreover, retaining a sufficiently small fixed proportion of the $d$-faces of complete complexes, depending on $C$ and $d$, while keeping their entire $(d-1)$-skeleta, gives complexes satisfying the lower-face inequalities and, since $q>d$, the equitability condition at sufficiently large densities. Taking disjoint unions of these complexes and using the upper bound from Theorem~\ref{Moon_Wendel} then yields $B_d'<\infty$. Since $B'_d=\limsup_{a\to\infty}b'_d(a)$ (where the limit is taken over rational $a$) and $q<d+1$, we can choose a rational $a_0$ large enough that $b'_d(a_0)\leq(1+\eta/4)B'_d$ and $(d+1)^{q+1}a_0^{q-d-1}\leq(\eta/4)B'_d$. Then, choose $s_0$ so that $a_0s_0$ is an integer and
\[\kappa'_d(s_0,a_0s_0)\leq(1+\eta/4)^2B'_d\,s_0a_0^{d+2}.\]

Let $X$ be a complex realizing $\kappa'_d(s_0,a_0s_0)$. Thus $X$ is $q$-equitable, $|X(d-1)|=s_0$, $|X(d)|=t_0:=a_0s_0$, and $\Cr_\iota(X)=\kappa'_d(s_0,a_0s_0)$.
We first bound $\Cr_\iota(X\centerdot m)$. Let $f\colon X\to\mathbb R^{2d}$ be a generic PL map with $\Cr_\iota(f)=\Cr_\iota(X)$. Choose a common subdivision on which $f\circ\varphi$ is linear, and let $g\colon X\centerdot m\to\mathbb R^{2d}$ be a sufficiently small generic PL perturbation of this map. Since disjoint compact sets have positive distance, the perturbation preserves the disjointness of every pair of lifted faces whose images under $f\circ\varphi$ are disjoint. Hence, a pair of facets of $X\centerdot m$ can intersect only if its image under $\varphi$ is an intersecting independent pair of $f$, or a pair of $d$-faces of $X$ sharing a vertex (the faces are allowed to coincide). Each such base pair has at most $m^{2d+2}$ lifted pairs, and at most $\sum_v\deg_d(v)^2$ pairs share a vertex. Since $X$ is $q$-equitable and $\sum_v\deg_d(v)=(d+1)t_0$,
\[
\begin{aligned}
\Cr_\iota(X\centerdot m)&\leq m^{2d+2}\Bigl(\Cr_\iota(X)+((d+1)a_0)^q\sum_{v\in X(0)}\deg_d(v)\Bigr)=m^{2d+2}\bigl(\Cr_\iota(X)+(d+1)^{q+1}a_0^{q+1}s_0\bigr)\\
&\leq\bigl((1+\eta/4)^2+\eta/4\bigr)B'_d\,s_0a_0^{d+2}m^{2d+2}\leq(1+\eta)B'_d\,s_0a_0^{d+2}m^{2d+2}.
\end{aligned}
\]

Let $s,t$ be as in the statement, and set 
\[m=\left\lceil(1+\eta)\frac{t}{a_0s}\right\rceil,\qquad k=\left\lfloor\frac{s}{s_0m^d}\right\rfloor.
\]
As $t/s$ and $s^{d+1}/t^d$ tend to infinity, both $m$ and $k$ tend to infinity. Thus, for sufficiently large $N$, we have that $1+\eta\leq ma_0s/t\leq1+2\eta$, $k\geq1/\eta$, and all terms of the form $O(k^{-1})$ below are smaller than $\eta$. Draw $k$ disjoint copies of $X\centerdot m$ in disjoint balls, and
denote their union by $Y$. Then
\[|Y(d-1)|=ks_0m^d\leq s,\qquad\frac{|Y(d)|}{t}=\frac{ks_0m^d}{s}\frac{ma_0s}{t}\geq\frac{k}{k+1}(1+\eta)\geq1.
\]

Delete $|Y(d)|-t$ faces of dimension $d$, retaining the entire $(d-1)$-skeleton. To supply the missing $s-ks_0m^d<s_0m^d$ faces of dimension $d-1$, take a complex $P$ generated by exactly that many distinct $d$-element subsets of a new vertex set. These subsets can be chosen on at most $\lceil(d!s_0)^{1/d}m\rceil+d-1=O(m)$ vertices. Consequently, $|P(i)|=O(m^{i+1})$ for $0\leq i\leq d-1$, and $P$ has no $d$-faces. Add $P$ in a separate ball and denote
the resulting complex by $Z$. We have $|Z(d-1)|=s$, $|Z(d)|=t$, and $\Cr_\iota(Z)\leq k\Cr_\iota(X\centerdot m)$.

We claim that $Z$ is admissible, since deletion does not increase $d$-degrees, and the new vertices have $d$-degree zero. Since $X$ is $q$-equitable and $q>d$,
\[\frac{\max_v\deg_d^Z(v)}{((d+1)t/s)^q}\leq m^{d-q}\left(\frac{ma_0s}{t}\right)^q\leq(1+2\eta)^qm^{d-q}=o(1).
\]
Thus $Z$ is $q$-equitable for sufficiently large $N$. Moreover, $|Z(i)|=km^{i+1}|X(i)|+|P(i)|$ for $0\leq i\leq d-1$, so the lower-face inequalities for $X$ give
\[
\frac{|Z(i-1)|}{|Z(i)|}\leq\frac{1+O(k^{-1})}{m}\frac{|X(i-1)|}{|X(i)|}<\frac{1+O(k^{-1})}{1+\eta}\,C\frac{s}{t}<C\frac{s}{t} .
\] for all $1\leq i\leq d-1.$
The case $i=0$, with $|Z(-1)|=1$, follows from \[\frac{|Z(-1)|}{|Z(0)|}\leq\frac{1}{km|X(0)|}=o(s/t).\]
All implicit constants depend only on $X$ and $d$.

We can therefore apply the definition of $\kappa'_d$ to $Z$. Using the preceding estimate
\[\Cr_\iota(X\centerdot m)\leq(1+\eta)B'_d s_0a_0^{d+2}m^{2d+2},\] we obtain
\[
\begin{aligned}
\kappa'_d(s,t)\frac{s^{d+1}}{t^{d+2}}
&\leq k\Cr_\iota(X\centerdot m)\frac{s^{d+1}}{t^{d+2}}\\
&\leq(1+\eta)B'_d\frac{ks_0m^d}{s}\left(\frac{ma_0s}{t}\right)^{d+2}\\
&\leq(1+\eta)(1+2\eta)^{d+2}B'_d<(1+\varepsilon)B'_d.
\end{aligned}
\]
\end{proof}

\section{Random geometric complexes}
\label{sec:random_geometric_complexes}

The basic idea of this section is similar to the one appearing in~\cite{czabarka2019some}, but we execute it in higher dimensions. We now introduce random geometric complexes on the round sphere. Fix $d\geq1$ and let $\mathcal A\subseteq(\mathbb S^{2d})^{d+1}$ be a Borel set which projects onto each factor, is invariant under permutations of the factors, and is invariant under the diagonal action of $\operatorname{SO}(2d+1)$. For $0\leq k\leq d$, define $\mathcal A_k$ to be its projection onto any $k+1$ factors (which we assume is Borel too).

Choose $n$ points $X_1,\ldots,X_n$ independently and uniformly on $\mathbb S^{2d}$. The associated \emph{random geometric complex} $X_n=X(n,\mathcal A)$ has $\{i_0,\ldots,i_k\}$ as a face precisely when $(X_{i_0},\ldots,X_{i_k})\in\mathcal A_k$. The projection condition makes this a simplicial complex: every subtuple of an admissible tuple is admissible. Here, an admissible lower-dimensional face need not extend to a $d$-face among the sampled vertices. Every face is drawn as the geodesic spherical simplex spanned by its vertices, and we denote this drawing by $f_n$.

Let $\nu_d$ denote spherical $d$-volume and let $s_d=\nu_d(\mathbb S^d)$. For independent uniform points $X_0,\ldots,X_d$, write $\sigma:=\sph(X_0,\ldots,X_d)$ and put
\[
P_k:=\mathbb P\bigl((X_0,\ldots,X_k)\in\mathcal A_k\bigr),\qquad q_d(\sigma):=\frac{\nu_d(\sigma)}{s_d},\qquad m_d:=\mathbb E\bigl[q_d(\sigma)\mathbf1_{\{(X_0,\ldots,X_d)\in\mathcal A\}}\bigr].
\]
 Thus $m_d$ is the expected normalized volume, counted as zero when the simplex is not admitted. We use the following form of Crofton's formula.

\begin{lemma}\label{lem:crofton_random_geometric}
Let $\mu$ be a rotation-invariant probability measure on nondegenerate geodesic $d$-simplices in $\mathbb S^{2d}$ so that sampling from $\mu$ almost surely results in a simplex contained in an open hemisphere. If $\sigma,\sigma'$ are independent with law $\mu$, then
\[
\mathbb P(\sigma\cap\sigma'\neq\emptyset)=2\left(\frac{\mathbb E_\mu[\nu_d(\sigma)]}{s_d}\right)^2.
\]
\end{lemma}

\begin{proof}
The equatorial $d$-spheres spanned by $\sigma$ and $\sigma'$ are independent and uniformly distributed. Almost surely they intersect in two antipodal points $x,-x$. Conditional on these two spheres, rotation invariance implies that the probability that $x$ belongs to a given one of these simplices is $\mathbb E_\mu[\nu_d(\sigma)]/s_d$. The simplices remain conditionally independent, so the probability that both contain $x$ is the square of this quantity. The same holds for $-x$. Since neither simplex contains an antipodal pair, the two events are disjoint.
\end{proof}

\begin{proposition}\label{prop:random_geometric_master}
Suppose that $P_d>0$. For the natural drawings $f_n$ of $X_n=X(n,\mathcal A)$, the normalized crossing numbers converge in probability to
\[
B_d(\mathcal A):=\frac{(d+1)^d}{d!}\,m_d^2\,\frac{P_{d-1}^{d+1}}{P_d^{d+2}}
\] as $n$ grows.
There are deterministic realizations with the same limit, and $B_d\leq B_d(\mathcal A)$.
\end{proposition}

\begin{proof}
For every $0\leq k\leq d$, we have $\mathbb E|X_n(k)|=P_k\binom n{k+1}$. Conditional on being admitted, a random $d$-simplex has a rotation-invariant distribution. The preceding lemma therefore gives probability $2m_d^2$ that two independent $(d+1)$-tuples are both admitted and their simplices intersect. Almost surely, incident faces have no intersections outside their common face, and independent faces intersect in at most one point. Counting unordered pairs gives
\[
\mathbb E\Cr(f_n)=\binom n{d+1}\binom{n-d-1}{d+1}m_d^2.
\]

The face counts and the crossing count, divided respectively by $\binom n{k+1}$ and $\binom n{2d+2}$, are U-statistics with bounded kernels. Their variances tend to zero, since only pairs of tuples sharing a vertex can contribute to the covariance. Consequently,
\[
\Cr(f_n)\frac{|X_n(d-1)|^{d+1}}{|X_n(d)|^{d+2}}\rightarrow B_d(\mathcal A)
\]
in probability. We note that the event where $|X_n(d)|=0$ has vanishing probability, so it can simply be ignored (we set the above ratio to be equal to $0$ whenever this occurs).
We may choose a deterministic realization for each $n$ for which these convergences hold. Its face counts satisfy $|X_n(d-1)|\sim P_{d-1}n^d/d!$ and $|X_n(d)|\sim P_dn^{d+1}/(d+1)!$.

For completeness, these realizations give the last assertion by the disjoint-union argument used in Lemma~\ref{lem:dense_seeds}. Write $a_n=|X_n(d)|/|X_n(d-1)|\sim cn$, where $c=P_d/((d+1)P_{d-1})>0$. For fixed $\eta>0$ and large $a$, choose $n=\lceil(1+\eta)a/c\rceil$. Then $a_n\geq a$. Taking arbitrarily many disjoint PL copies of the drawing, and using that $\kappa_d(s,t)$ is defined over complexes with at least $t$ faces of dimension $d$, gives
\[b_d(a)\leq\frac{\Cr(f_n)}{a^{d+2}|X_n(d-1)|}=\left(\frac{a_n}{a}\right)^{d+2}\Cr(f_n)\frac{|X_n(d-1)|^{d+1}}{|X_n(d)|^{d+2}}.\]
As $a\to\infty$, we have $a_n/a\to1+\eta$, so $B_d\leq(1+\eta)^{d+2}B_d(\mathcal A)$. Letting $\eta\to0$ proves the assertion.
\end{proof}

The proposition concerns the crossing number of the specified drawing. It gives an upper bound for the minimum crossing number of the complex, but does not assert that the natural drawing is optimal. The use of disjoint PL copies also does not give the same bound for $B_{d,\mathbb S}$. For the small-threshold construction below, we will obtain a spherical bound by a separate Euclidean construction.

\subsection{Edge-threshold complexes}
\label{subsec:random-geometric-flag}

For $\alpha\in(0,\pi]$, we construct a simplicial complex by sampling $n$ independent uniform points of $\mathbb S^{2d}$ and including a set of at most $d+1$ of them as a face whenever its diameter is at most $\alpha$. We call this random complex the \emph{edge-threshold complex} and denote it by $X_{n,\alpha,d}$. This is the preceding model with $\mathcal A$ consisting of tuples of diameter at most $\alpha$.

In this subsection, $P_k(\alpha)$ and $m_d(\alpha)$ refer to this model. Proposition~\ref{prop:random_geometric_master} gives
\[
B_d\leq C_d^{\mathrm{edge}}(\alpha):=\frac{(d+1)^d}{d!}\,m_d(\alpha)^2\,\frac{P_{d-1}(\alpha)^{d+1}}{P_d(\alpha)^{d+2}}.
\]

\subsubsection{The graph case, \texorpdfstring{$d=1$}{d=1}}

The model is the usual random geometric graph on $\mathbb S^2$ considered in~\cite{czabarka2019some}. We have $P_0=1$, $P_1(\alpha)=(1-\cos\alpha)/2$, and $m_1(\alpha)=(\sin\alpha-\alpha\cos\alpha)/(4\pi)$. Therefore
\[
C_1^{\mathrm{edge}}(\alpha)=\frac{(\sin\alpha-\alpha\cos\alpha)^2}{\pi^2(1-\cos\alpha)^3},\qquad\lim_{\alpha\to0 }C_1^{\mathrm{edge}}(\alpha)=\frac8{9\pi^2}.
\]
This recovers the random geometric graph bound from~\cite{czabarka2019some}.

\subsubsection{The case \texorpdfstring{$d=2$}{d=2}}
\label{subsubsec:edge-d2}

We now sample vertices on $\mathbb S^4$ and keep a triangle precisely when its three edges have length at most $\alpha$. It is convenient to use the edge lengths as coordinates. For $r,s,t\in[0,\pi]$, put
\begin{equation*}
W(r,s,t):=\det\begin{pmatrix}1&\cos t&\cos s\\\cos t&1&\cos r\\\cos s&\cos r&1\end{pmatrix}
=1-\cos^2r-\cos^2s-\cos^2t+2\cos r\cos s\cos t.
\end{equation*}
The condition $W>0$ characterizes nondegenerate triples of this kind.

\begin{lemma}\label{lem:distance-density}
The pairwise distances of three independent uniform points of $\mathbb S^4$ have density
\[
\frac9{8\pi}\sin r\sin s\sin t\sqrt{W(r,s,t)_+}
\]
on $[0,\pi]^3$, where $x_+:=\max\{x,0\}$ for all $x\in\mathbb R$.
\end{lemma}

\begin{proof}
Fix the first point at the north pole, and let $r,s$ be its distances to the other two points. They are independent with density $\frac34\sin^3r$. The cosine $z$ of the angle between their tangent directions has density $\frac2\pi\sqrt{1-z^2}$, independently of $r,s$. The spherical cosine rule gives $\cos t=\cos r\cos s+\sin r\sin s\,z$, so $|dz|=\sin t\,dt/(\sin r\sin s)$ and $1-z^2=W/(\sin r\sin s)^2$. Multiplying the densities and the Jacobian proves the claim.
\end{proof}

For $W>0$, the spherical area is
\[
A(r,s,t)=2\operatorname{atan2}\bigl(\sqrt W,1+\cos r+\cos s+\cos t\bigr),
\]
where $\operatorname{atan2}(y,x)\in(0,\pi)$ for $y>0$. This choice of branch is necessary when the second argument is negative. Throughout the rest of this section, we set $A=0$ outside $W>0$; this does not affect the computations below. Writing $\Omega_\alpha=[0,\alpha]^3$ and $D=\sin r\sin s\sin t\sqrt{W_+}$, we obtain
\[
P_1(\alpha)=\frac12-\frac34\cos\alpha+\frac14\cos^3\alpha,\]
\[
P_2(\alpha)=\frac9{8\pi}\int_{\Omega_\alpha}D(r,s,t)\,dr\,ds\,dt,\] and
\[m_2(\alpha)=\frac9{32\pi^2}\int_{\Omega_\alpha}D(r,s,t)A(r,s,t)\,dr\,ds\,dt.\]
At $\alpha=\pi$, these give $P_1=P_2=1$ and $m_2=1/8$, and hence the complete-complex bound $C_2^{\mathrm{edge}}(\pi)=9/128$.

The matrix appearing in the definition of $W$ is the Gram matrix $U$ of the three sampled points, viewed as unit vectors in $\mathbb R^5$: its entries are their pairwise inner products, which equal the cosines of their spherical distances. Changing variables from the distances to their cosines in Lemma~\ref{lem:distance-density} cancels the three sine factors, so $U$ has density proportional to $\det(U)^{1/2}$ with respect to Lebesgue measure on its off-diagonal entries. More generally, for $0\leq k\leq d$, if $X_0,\ldots,X_k$ are independent uniform points of $\mathbb S^{2d}$, their Gram matrix $U=(\langle X_i,X_j\rangle)_{0\leq i,j\leq k}$ has density proportional to $\det(U)^{(2d-k-1)/2}$ on the positive definite matrices with diagonal entries equal to one. This gives an analogous description of the joint distribution of the pairwise distances in every dimension.

We next let $\alpha\to0$. Expanding the definition of $W$ gives
\[
W(r,s,t)=\frac14(2r^2s^2+2s^2t^2+2t^2r^2-r^4-s^4-t^4)+O((r+s+t)^6).
\]
Let $a(\rho,\sigma,\tau)$ be the Euclidean area of a triangle with the indicated sides, and set it equal to zero when the triangle inequalities fail. After substituting $(r,s,t)=\alpha(\rho,\sigma,\tau)$, Heron's formula gives $\alpha^{-2}\sqrt{W_+}\to2a$ and $\alpha^{-2}A\to a$ almost everywhere on $[0,1]^3$. These rescaled quantities are bounded for small $\alpha$, so dominated convergence yields
\[
P_1(\alpha)\sim\frac3{16}\alpha^4,\qquad P_2(\alpha)\sim\frac9{4\pi}I\alpha^8,\qquad m_2(\alpha)\sim\frac9{16\pi^2}J\alpha^{10},
\]
where $I=\int_{[0,1]^3}\rho\sigma\tau\,a\,d\rho\,d\sigma\,d\tau$ and $J=\int_{[0,1]^3}\rho\sigma\tau\,a^2\,d\rho\,d\sigma\,d\tau$.

\begin{lemma}\label{lem:IJ}
We have $I=(8\pi-9\sqrt3)/512$ and $J=53/12800$.
\end{lemma}

\begin{proof}
By symmetry and homogeneity, fixing the largest side gives
\[I=\frac38\int_Txy\,a(1,x,y)\,dx\,dy\] and
\[J=\frac3{10}\int_Txy\,a(1,x,y)^2\,dx\,dy,\] where
$T=\{(x,y)\in[0,1]^2:x+y\geq1\}$.
Put $u=x+y$ and $v=x-y$. Heron's formula becomes
$16a(1,x,y)^2=(u^2-1)(1-v^2)$, and consequently
\[
I=\frac3{256}\int_1^2\int_{u-2}^{2-u}(u^2-v^2)\sqrt{(u^2-1)(1-v^2)}\,dv\,du,\]
\[J=\frac3{1280}\int_1^2\int_{u-2}^{2-u}(u^2-v^2)(u^2-1)(1-v^2)\,dv\,du.\]

Evaluating these elementary integrals gives the stated values. In the first integral one may first integrate in $u$, using primitives of $\sqrt{u^2-1}$ and $u^2\sqrt{u^2-1}$; the second integrand is polynomial.
\end{proof}

Since $C_2^{\mathrm{edge}}=\frac92m_2^2P_1^3/P_2^4$, the powers of $\alpha$ cancel and we obtain
\begin{equation*}\label{eq:C2-small-alpha}
\lim_{\alpha\to0}C_2^{\mathrm{edge}}(\alpha)
=\frac3{8192}\frac{J^2}{I^4}
=\frac{3\cdot53^2\cdot2^5}{5^4(8\pi-9\sqrt3)^4}
=0.05199595\ldots.
\end{equation*}

To obtain the spherical midrange bound, we first construct arbitrarily large Euclidean complexes with a prescribed asymptotic ratio of triangles to edges, and then project their drawings onto the sphere.

\begin{proposition}\label{prop:edge_spherical_midrange}
The spherical midrange constant satisfies
\[
B_{2,\mathbb S}\leq\frac{3\cdot53^2\cdot2^5}{5^4(8\pi-9\sqrt3)^4}<0.052.
\]
\end{proposition}

\begin{proof}
Take a homogeneous Poisson process of intensity $\lambda>0$ in $[0,L]^4$. Join pairs at distance at most one and fill every triangle using affine faces. Let $V=L^4$. As $L\to\infty$, the edge count, triangle count, and number of crossings, divided by $V$, converge in probability to
\[
\frac{\pi^2}{4}\lambda^2,\qquad
\frac{8\pi^3}{3}I\lambda^3,\qquad
\frac{32\pi^6}{27}J^2\lambda^6,
\]
respectively. To compute these constants, consider a Poisson process of intensity $\lambda$ on the sphere of radius $\alpha^{-1}$, with distance threshold one. Its volume is $V_\alpha=s_4\alpha^{-4}$, where $s_4=8\pi^2/3$. Applying the spherical expectation formulas with angular threshold $\alpha$ and mean vertex count $\lambda V_\alpha$, then dividing by $V_\alpha$, gives the displayed constants as $\alpha\to0$. These limits equal the Euclidean intensities by dominated convergence: after fixing one vertex, every contributing configuration lies in a uniformly bounded neighborhood. Project this neighborhood from the center of the sphere onto the tangent hyperplane at the fixed vertex. This sends geodesic simplices to affine simplices and preserves their intersections. As $\alpha\to0$, distances and volume elements in these coordinates converge to their Euclidean counterparts. Expected boundary contributions in the box are $O(L^3)$. The variances are $O(V)$, since contributions assigned to unit cubes at a sufficiently large fixed distance are independent and have bounded second moments. This proves the asserted convergence.

Fix $a>0$ and $\eta>0$, and choose $\lambda=3(1+\eta)a/(32\pi I)$. The triangle-to-edge ratio then tends to $(1+\eta)a$. Choose deterministic realizations with the preceding limits; for all sufficiently large $L$, they have at least $\lfloor a|X(1)|\rfloor$ triangles. Central projection from $\mathbb R^4$ onto an open hemisphere of $\mathbb S^4$, given by $x\mapsto(x,1)/\sqrt{1+\|x\|^2}$, takes affine simplices to geodesic simplices and preserves their intersections. These are therefore admissible spherical drawings, and applying the preceding limits in the definition of $b_{2,\mathbb S}(a)$ gives
\[b_{2,\mathbb S}(a)\leq(1+\eta)^4\,\frac{3J^2}{8192I^4}.\]
Letting $\eta\to0$ and then taking the limsup as $a\to\infty$ proves the assertion.\end{proof}

Numerically evaluating the integrals defining $P_2(\alpha)$ and $m_2(\alpha)$, and substituting the results together with the explicit formula for $P_1(\alpha)$ into $C_2^{\mathrm{edge}}(\alpha)=\frac92m_2(\alpha)^2P_1(\alpha)^3/P_2(\alpha)^4$, suggests that this expression is minimized near $\alpha=0.95843$. At this threshold, $P_1(\alpha)\approx0.11638$ and $C_2^{\mathrm{edge}}(\alpha)\approx0.0516002$, which is below the small-$\alpha$ limit. Since $B_2\leq C_2^{\mathrm{edge}}(\alpha)$ for every $\alpha\in(0,\pi]$, proving $B_2<0.0517$ only requires a certified upper estimate below $0.0517$ at a single threshold; locating or certifying a global minimum is unnecessary.

\subsection{Volume threshold complexes, Gale duality, and \texorpdfstring{$d\to\infty$}{d to infinity}}
\label{subsec:volume_threshold}

We now consider a second rotationally invariant model. For $\theta\in(0,1/2)$, let $\mathcal A_\theta\subseteq(\mathbb S^{2d})^{d+1}$ consist of the $(d+1)$-tuples whose spherical simplex has normalized volume $q=\nu_d/s_d$ at most $\theta$, where degenerate tuples are given $q=0$. Sample $N$ independent uniform points of $\mathbb S^{2d}$, and let $X_N^{\mathrm{vol}}(\theta):=X(N,\mathcal A_\theta)$ be the associated random geometric complex, with natural drawing $f_N$. Any $d$-tuple is the projection of an admissible tuple, obtained by repeating one of its points, so $X_N^{\mathrm{vol}}(\theta)$ has a complete $(d-1)$-skeleton; it keeps exactly the $d$-simplices of normalized volume at most $\theta$.

Let $q$ be the normalized volume of the spherical simplex spanned by $d+1$ independent uniform points of $\mathbb S^{2d}$, and put
\[
P(\theta)=\mathbb P(q\leq\theta),\qquad m(\theta)=\mathbb E[q\mathbf1_{\{q\leq\theta\}}],\qquad \bar q=\mathbb E q=2^{-(d+1)}.
\]
The last identity follows by independently changing the signs of the vertices: the resulting $2^{d+1}$ cones partition their linear span, and all have the law of the original one.

\begin{proposition}\label{prop:volume_threshold_bound}
For every fixed $\theta\in(0,1/2)$, the normalized crossing numbers of $f_N$ converge in probability to
\[
C_d^{\mathrm{vol}}(\theta):=\frac{(d+1)^d}{d!}\frac{m(\theta)^2}{P(\theta)^{d+2}}, \quad\text{and}\quad B_d\leq C_d^{\mathrm{vol}}(\theta).
\]
\end{proposition}

\begin{proof}
This is Proposition~\ref{prop:random_geometric_master} with $\mathcal A=\mathcal A_\theta$. The set $\mathcal A_\theta$ is invariant under permutations and rotations, and its projection onto any $d$ factors is the whole product, so $P_{d-1}=1$. Moreover $P_d=P(\theta)>0$ and $m_d=m(\theta)$.
\end{proof}

We compare this with the uniform geodesic drawing of the complete complex from Theorem~\ref{Moon_Wendel}, whose constant is
\[
C_d^{\mathrm{comp}}:=\frac{(d+1)^d}{2^{2d+2}d!}.
\]
In particular, $B^\iota(\Delta_{(d)})\leq B^{\mathbb S}(\Delta_{(d)})\leq C_d^{\mathrm{comp}}$. The ratio of the two construction constants is
\[
\frac{C_d^{\mathrm{vol}}(\theta)}{C_d^{\mathrm{comp}}}=\frac{(m(\theta)/\bar q)^2}{P(\theta)^{d+2}}.
\]
Since $q<1/2$ almost surely, $P(\theta)\to1$ and $m(\theta)\to\bar q$ as $\theta\uparrow1/2$, so $C_d^{\mathrm{vol}}(\theta)\to C_d^{\mathrm{comp}}$: removing no simplex recovers the complete complex. The next theorem shows that removing the simplices of largest volume always helps.

\begin{theorem}\label{thm:volume_threshold_strict}
For every $d\geq1$, we have $\inf_{0<\theta<1/2}C_d^{\mathrm{vol}}(\theta)<C_d^{\mathrm{comp}}$, and consequently $B_d<C_d^{\mathrm{comp}}$.
\end{theorem}

\begin{proof}
Write $\varepsilon_\theta=\mathbb P(q>\theta)$ and $r_\theta=\mathbb E[q\mid q>\theta]$. Simplicial cones can have normalized solid angle arbitrarily close to $1/2$, for instance when $d$ of the vertices nearly positively span a hyperplane of their linear span and the last vertex lies slightly off it. By continuity, such configurations have neighborhoods of positive probability, so $\varepsilon_\theta>0$ for $\theta<1/2$. Moreover $\varepsilon_\theta\to0$ as $\theta\uparrow1/2$, because $q<1/2$ almost surely, and $r_\theta\to1/2$, because $\theta<r_\theta<1/2$. Since $P(\theta)=1-\varepsilon_\theta$ and $m(\theta)=\bar q-\varepsilon_\theta r_\theta$, we have
\[
\frac{C_d^{\mathrm{vol}}(\theta)}{C_d^{\mathrm{comp}}}=1+\varepsilon_\theta\left(d+2-\frac{2r_\theta}{\bar q}\right)+O(\varepsilon_\theta^2).
\]
The coefficient tends to $d+2-2^{d+1}<0$, proving the claim.
\end{proof}

We next show that the improvement over $C_d^{\mathrm{comp}}$ is exponential in $d$. Let $Q^*$ have the size-biased distribution of $q$, defined by $\mathbb P(Q^*\in E)=\bar q^{-1}\mathbb E[q\mathbf1_{\{q\in E\}}]$ for measurable $E$. Then
\begin{equation}\label{eq:sizebiased-volume}
\frac{m(\theta)}{\bar q}=\mathbb P(Q^*\leq\theta).
\end{equation}
We will find a threshold which keeps almost all simplices but retains only an exponentially small part of this size-biased distribution.

A $p\times k$ \emph{standard Gaussian matrix} is a random matrix whose $pk$ entries are independent standard Gaussian variables. The following Gaussian calculation will be useful.

\begin{lemma}\label{lem:gaussian_kernel_quadratic}
Let $G$ be an $M\times(M+1)$ standard Gaussian matrix. Almost surely its kernel is one-dimensional; let $\lambda$ be a unit vector spanning it. Conditional on the kernel line, for every deterministic $x\in\mathbb R^{M+1}$,
\[
\|Gx\|^2\stackrel{d}{=}\|P_{\lambda^\perp}x\|^2\chi^2_{M+1}.
\]
Here $\chi_{M+1}$ denotes the Euclidean norm of a vector of $M+1$ independent standard Gaussian variables, sampled independently of the kernel line, so that $\chi_{M+1}^2$ has the chi-square distribution with $M+1$ degrees of freedom. Moreover, the kernel line is uniformly distributed and independent of $GG^\top$.
\end{lemma}

\begin{proof}
Write the thin QR decomposition of $G^\top$ as $G^\top=QR$. Here $Q$ is an $(M+1)\times M$ matrix with orthonormal columns, so $Q^\top Q=I_M$, and $R$ is an $M\times M$ upper triangular matrix with positive diagonal entries. The orthonormal frame formed by the columns of $Q$ is uniform and independent of $R$. Conditional on its column space $\lambda^\perp$, it is still a uniform orthonormal frame of that space. In particular, the kernel line, which is the orthogonal complement of the column space of $Q$, is uniformly distributed and independent of $R$, hence of $GG^\top=R^\top R$. Thus $Q^\top x$ is uniform on the sphere in $\mathbb R^M$ of radius $r=\|P_{\lambda^\perp}x\|$, independently of $R$. If $U$ is an independent uniform $M\times M$ orthogonal matrix and $e_1$ is the first standard basis vector of $\mathbb R^M$, then, conditional on the kernel line,
\[
\|Gx\|^2\stackrel d=r^2e_1^\top U^\top RR^\top Ue_1.
\]
The matrices $RR^\top$ and $R^\top R$ have the same eigenvalues, and $R^\top R=GG^\top$ has an orthogonally invariant distribution, independent of the kernel line. Consequently, $U^\top RR^\top U$ has the distribution of $GG^\top$. Finally, $(GG^\top)_{11}$ is the squared norm of a row of $G$, and therefore has distribution $\chi^2_{M+1}$.
\end{proof}

\begin{lemma}\label{lem:gale_volume_separation}
Put $\theta_d=2^{-(d+1)}e^{d/400}$. For every $d\geq130$,
\[
\mathbb P(q>\theta_d)\leq e^{-d/400},\qquad\mathbb P(Q^*\leq\theta_d)\leq4e^{-d/230}.
\]
\end{lemma}

\begin{proof}
The first bound is Markov's inequality: $\mathbb P(q>\theta_d)\leq\bar q/\theta_d=e^{-d/400}$.

For the second, let $A,B$ be independent $(2d+1)\times(d+1)$ standard Gaussian matrices. Normalizing their columns gives independent uniform spherical vertices, without changing the positive cones or their normalized solid angles. Let $q(A)$ denote the normalized solid angle of the cone generated by the columns of $A$, that is, the normalized volume of the corresponding spherical simplex, and let $\mathcal E$ be the event that the cones generated by $A$ and by $B$ meet nontrivially. The argument of Lemma~\ref{lem:crofton_random_geometric}, with the first simplex fixed, gives $\mathbb P(\mathcal E\mid A)=2\bar q\,q(A)$ and $\mathbb P(\mathcal E)=2\bar q^2$. Therefore the law of $A$ conditioned on $\mathcal E$ is its size-biased law. Denoting the ordinary law by $\mathbb P$ and this conditioned law by $\mathbb P^*$, we have
\begin{equation}\label{eq:size_bias_likelihood}
\frac{d\mathbb P^*}{d\mathbb P}(A)=\frac{q(A)}{\bar q}=:L(A).
\end{equation}
In particular, under $\mathbb P^*$ the variable $q(A)$ has the law of $Q^*$, so
\[
\mathbb P(Q^*\leq\theta_d)=\mathbb P^*(q(A)\leq\theta_d)=\mathbb P^*(L\leq e^{d/400}).
\]

The $(2d+1)\times(2d+2)$ matrix $G=(A\ B)$ has a one-dimensional kernel almost surely. The event $\mathcal E$ says precisely that its unit kernel vector can be written as $\lambda=(a,-b)$, where $a,b\in\mathbb R^{d+1}$ have all coordinates positive. In Gale-dual terms, crossing is a sign condition on the kernel of $G$. For Gaussian vertices, the kernel line is uniform and independent of $GG^\top$ (Lemma~\ref{lem:gaussian_kernel_quadratic}). This independence is what what allows us to access the size-biased law: conditioning on $\mathcal E$ prescribes only the signs of the kernel vector. Therefore, under the joint law of $(A,B)$ conditioned on $\mathcal E$,
\[
(a,b)\stackrel d=\frac{(|Z_1|,\ldots,|Z_{2d+2}|)}{(\sum_{j=1}^{2d+2}Z_j^2)^{1/2}},
\]
where the $Z_j$ are independent standard Gaussians.

Set $u=(1,\ldots,1)/\sqrt{d+1}\in\mathbb R^{d+1}$ and $R(A)=\|Au\|^2/(2d+1)$. Under $\mathbb P$, we have $R(A)\stackrel d=\chi_{2d+1}^2/(2d+1)$. Conditional on the kernel line, Lemma~\ref{lem:gaussian_kernel_quadratic}, applied to $x=(u,0)\in\mathbb R^{d+1}\times\mathbb R^{d+1}$, gives
\[
R(A)\stackrel d=(1-\langle a,u\rangle^2)\frac{\chi^2_{2d+2}}{2d+1}.
\]
Here the chi-square variable is independent of the kernel line. The identity therefore remains valid under conditioning on $\mathcal E$, which depends only on that line.

On the event that $\sum_{i=1}^{d+1}|Z_i|\geq0.7(d+1)$ and $\sum_{j=1}^{2d+2}Z_j^2\leq2.2(d+1)$, we have
\[
\langle a,u\rangle^2=\frac{(\sum_{i=1}^{d+1}|Z_i|)^2}{(d+1)\sum_{j=1}^{2d+2}Z_j^2}\geq\frac{49}{220}>\frac15.
\]
Write $h(x)=x-1-\log x$. For a chi-square variable with $k$ degrees of freedom, the Chernoff bound at $kx$ is $e^{-kh(x)/2}$, for the upper tail when $x>1$ and the lower tail when $0<x<1$. The function $\sum_{i=1}^{d+1}|Z_i|$ is $\sqrt{d+1}$-Lipschitz with mean $(d+1)\sqrt{2/\pi}$. By Gaussian concentration, the first condition fails with probability at most $\exp\bigl(-(d+1)(\sqrt{2/\pi}-0.7)^2/2\bigr)\leq e^{-d/210}$. By the Chernoff bound, the second fails with probability at most $e^{-(d+1)h(11/10)}\leq e^{-d/214}$.

Let $\mathcal T=\{R(A)\leq22/25\}$. Since $h(22/25)>1/140$, the Chernoff bound gives
\[
\mathbb P(\mathcal T)\leq\exp\bigl(-(d+\tfrac12)h(22/25)\bigr)\leq e^{-d/140}.
\]
Under $\mathbb P^*$, the factor $1-\langle a,u\rangle^2$ is at most $4/5$ except on an event of probability at most $e^{-d/210}+e^{-d/214}$. Outside that event, the identity in law above bounds the conditional probability of $R(A)>22/25$ by
\[
\mathbb P\bigl(\chi^2_{2d+2}>\tfrac{11}{10}(2d+1)\bigr)\leq e^{-(d+1)h(x_d)},\qquad x_d=\frac{11}{10}\cdot\frac{2d+1}{2d+2},
\]
which is at most $e^{-d/230}$ for $d\geq130$. Hence
\[
\mathbb P^*(\mathcal T^c)\leq e^{-d/210}+e^{-d/214}+e^{-d/230}.
\]

Finally, the likelihood-ratio identity gives
\[
\begin{split}
\mathbb P^*(L\leq e^{d/400})&\leq\mathbb P^*(\mathcal T^c)+\mathbb P^*(\mathcal T\cap\{L\leq e^{d/400}\})\\
&=\mathbb P^*(\mathcal T^c)+\mathbb E\bigl[L\mathbf1_{\mathcal T\cap\{L\leq e^{d/400}\}}\bigr]\\
&\leq\mathbb P^*(\mathcal T^c)+e^{d/400}\,\mathbb P(\mathcal T)\leq4e^{-d/230},
\end{split}
\]
because $1/140-1/400>1/230$. This is the asserted bound for $Q^*$.
\end{proof}

\begin{proof}[Proof of Theorem~\ref{thm:volume_threshold_exponential}]
Let $d\geq2400$ and take $\theta_d$ from the preceding lemma. By the expression for the ratio of the construction constants and~\eqref{eq:sizebiased-volume},
\[
\frac{C_d^{\mathrm{vol}}(\theta_d)}{C_d^{\mathrm{comp}}}
=\frac{\mathbb P(Q^*\leq\theta_d)^2}{\mathbb P(q\leq\theta_d)^{d+2}}
\leq\frac{16e^{-d/115}}{(1-e^{-d/400})^{d+2}}
\leq e^{-(d+1)/200}.
\]
For the last inequality, put $y=e^{-d/400}$ and consider the logarithm of the ratio of the two sides, $\log16-d/115+(d+1)/200-(d+2)\log(1-y)$. It is negative at $d=2400$. It is decreasing for $d\geq2400$, since $-1/115+1/200<0$ and the derivative of $-(d+2)\log(1-y)$ is at most $y/(1-y)-(d+2)y/(400(1-y))<0$ there. Proposition~\ref{prop:volume_threshold_bound} therefore proves
\[
B_d\leq e^{-(d+1)/200}\frac{(d+1)^d}{2^{2d+2}d!}.
\]
\end{proof}

We note that the proof uses no direct estimate of the random solid angle. Conditioning two simplices to cross gives its size-biased distribution, and the resulting sign condition on the kernel vector makes the statistic $\|Au\|^2$ typically smaller under this conditioning. The two laws can therefore be distinguished with exponentially small error, which yields the required separation of their solid angles. Finally, Stirling's formula gives $C_d^{\mathrm{comp}}\sim\frac{e}{4\sqrt{2\pi d}}(e/4)^d$, so the construction improves the complete-skeleton bound by an additional exponential factor.

\paragraph{Acknowledgments}
We are indebted to Matthieu Fradelizi for several discussion on this topic and, in particular, for the calculation of the upper bound of $B_{2,\mathbb S}$. The first named author thanks Arnaud de Memsmay and Uli Wagner for helpful discussions. The third author is grateful to Larry Guth for multiple insightful conversations.

\paragraph{Disclosure of AI use.}

An LLM, after being given our calculations for $B_{2,\mathbb S}$, discovered the upper bound on $B_2$. Prompted with the suggestion to use Gale duality, an LLM found the proof Theorem~\ref{thm:volume_threshold_exponential}. We have verified every mathematical claim and assume full responsibility for the results presented in this paper.

\bibliographystyle{plain} 
\bibliography{crossing}

\end{document}